\documentclass[11pt]{article}%
\usepackage{amsmath,amssymb,amsthm,amsfonts,comment}
\usepackage{tikz-cd}
\usepackage{fullpage}
\usepackage{graphicx}
\usepackage{xypic}
\usepackage{hyperref}
\usepackage{pstricks}
\usepackage{array}
\usepackage{mathdots}
\usepackage{enumitem}
\allowdisplaybreaks
\newtheorem{algorithm}{Algorithm}
\newtheorem{theorem}{Theorem}
\newtheorem{lemma}{Lemma}
\newtheorem{example}[lemma]{Example}

\newtheorem{proposition}[lemma]{Proposition}

\newtheorem{remark}[lemma]{Remark}

\newtheorem{definition}[lemma]{Definition}
\def\k{\mathbb K}

\def\RR{\mathbb R}
\def\kk{\overline{\mathbb K}}

\def\mf{P}
\def\bb{b}
\def\aa{\mathbf a}
\def\hh{\mathbf h}
\def\uu{\mathbf u}
\def\ww{\mathbf w}
\def\bb{\mathbf b}
\def\wh{\widehat }

\def\rank{{\mathrm rank}}
\def\bez{{B\'ezout }}

\def\syz{\operatorname{syz}}

\newcommand{\beq}{\begin{equation}}
\newcommand{\eeq}{\end{equation}}

\newcommand\Blue{}
\newcommand\Red{}
\begin{document}

\title{Degree-optimal Moving Frames for Rational Curves
\thanks{The research was partially supported by the grant US NSF CCF-1319632.}}
\author{Hoon Hong
  \thanks{North Carolina State University (hong@ncsu.edu, zchough@ncsu.edu, iakogan@ncsu.edu).}
 \and Zachary Hough\footnotemark[2]
 \and Irina Kogan\footnotemark[2] 
 \and Zijia Li \thanks{Joanneum Research, Austria (zijia.li@joanneum.at).}
}
\maketitle

\begin{abstract}

A \emph{moving frame} at a rational curve is a basis of vectors 
moving along the  curve. When the rational curve is\ given parametrically by a row vector $\aa$ of  univariate polynomials, a  moving frame with important algebraic properties can be defined by the columns of an invertible polynomial matrix $P$, such that  $\aa P=[\gcd(\aa),0\ldots,0]$. 
A \emph{degree-optimal moving frame} has column-wise  minimal degree,  where the degree of a column is defined to be   the maximum of the degrees of its components.  Algebraic moving frames are closely related to the univariate versions of the celebrated Quillen-Suslin problem,   effective Nullstellensatz problem,  and syzygy module problem.
However, this paper appears to be the first devoted to finding an  efficient algorithm for constructing  a {degree-optimal} moving frame, a property desirable in various applications. We compare our algorithm with other possible approaches, based on already available algorithms, and show that it is more efficient. We  also establish several new theoretical results  concerning the degrees of an  optimal moving frame and its components. In addition, we show that any deterministic  algorithm for computing a degree-optimal algebraic moving frame can be augmented so that it assigns   a degree-optimal  moving frame in a $GL_n(\k)$-equivariant manner.  This crucial property of classical {geometric moving frames},  in combination with the algebraic properties, can be exploited in various problems.
\end{abstract}



\section{Introduction}


Let $\mathbb{K}[s]$ denote a  ring of {\Red univariate} polynomials over a field
$\mathbb{K}$ and let $\mathbb{K}[s]^{n}$ denote the set of row vectors of length
$n$ over $\mathbb{K}$. Let $GL_{n}(\mathbb{K}[s])$ denote the set
of invertible $n\times n$ matrices over $\mathbb{K}[s]$, or equivalently, the set of matrices whose columns are \emph{point-wise} linearly independent {\Red over the algebraic closure $\kk$.} 
 
A nonzero row vector $\mathbf{a}\in\mathbb{K}[s]^{n}$ defines a parametric curve in $\k^n$.  The columns of a matrix  $P\in   GL_{n}(\mathbb{K}[s])$  assign a basis of vectors in $\k^n$ at each point of the curve. In other words, the columns of the matrix can be viewed as a coordinate system, or a frame, that moves along the curve. To be of interest, however, such assignment should not be arbitrary, but instead be related to the curve in a meaningful way.  In this paper,  we require that  $\mathbf{a}\, P=[\gcd(\mathbf{a}),0,\dots,0]$, where  $\gcd(\aa)$ is the monic greatest common divisor of the components of  $\aa$. We will call a matrix $P$ with the above property an \emph{algebraic moving frame at $\aa$}. We  observe that  for any
nonzero monic polynomial $\lambda(s)$, a moving frame at $\mathbf{a}$ is also a moving frame at $\lambda
\mathbf{a}$. Therefore, we can obtain an equivalent construction in the projective
space $\mathbb{PK}^{n-1}$ by considering only polynomial vectors $\aa$ such that $\gcd(\aa)=1$.
Then $\mf$ can be thought of
as an element of $PGL_{n}(\mathbb{K}[s])=GL_{n}(\mathbb{K}[s])/{cI}$, where
$c\neq0\in\mathbb{K}$ and $I$ is an identity matrix. A canonical map of $\aa$ to any of the affine
subsets $\mathbb{K}^{n-1}\subset \mathbb{PK}^{n-1}$ produces a rational curve in $\k^n$, and $P$ assigns a projective moving frame at $\aa$. This paper is devoted to  \emph{degree-optimal} algebraic moving frames -- frames that column-wise have minimal degrees, where the degree of a column is defined to be   the maximum of the degrees of its components (see Definitions~\ref{def-basic} and~\ref{def-omf}).

Algebraic moving frames  appeared in a number of important proofs and constructions under a variety of names.  For example, in the constructive proofs of the celebrated  Quillen-Suslin  theorem \cite{fitchas-galligo-1990},  \cite{logar-sturmfels-1992},  \cite{caniglia-1993}, \cite{park-woodburn-1995},  \cite{lombardi-yengui-2005}, \cite{fabianska-quadrat-2007}, given a polynomial \emph{unimodular}  $m\times n$ matrix ${\bf A}$, one constructs a unimodular matrix $P$ such that ${\bf A}P=[I_m,\bf{0}]$, where $I_m$ is an $m\times m$ identity matrix. In the univariate case  with $m=1$, the matrix $P$ is an algebraic moving frame.  However,  the above works  were not concerned  with the problem of finding  $P$ of optimal degree for every input  ${\bf A}$. Under the same assumptions  on ${\bf A}$,  a \emph{minimal multiplier}, defined in Section  3 of \cite{beckermann-labahn-villard-2006}, is a {degree-optimal} algebraic moving frame. However, the paper \cite{beckermann-labahn-villard-2006} was not concerned with constructing minimal multipliers, and a direct algorithm for computing them was not introduced. 
In Section~\ref{sect-compare}, we discuss a two-step approach, consisting of constructing a non-optimal moving frame and then performing a degree-reduction procedure. We show that it is less efficient than  the direct approach developed in the current paper.
An alternative direct approach for computing degree-optimal moving frames is in the dissertation of the second author (\cite{hough-2018}, Sections 5.9 and 5.10)  This approach is based on computing the term-over-position (TOP) Gr\"obner basis of a certain module over $\k[s]$, and when standard  TOP Gr\"obner basis algorithms for modules are employed, it is less efficient than the algorithm in the current paper. Optimizations, based on the structure of the particular problem, are possible and are the subject of a forthcoming publication.


A very important area of applications where, according to our preliminary studies,   utilization of degree-optimal moving frames is beneficial,   is  the control theory.  In particular,  the use of degree-optimal frames can lower differential  degrees of ``flat outputs'' (see, for instance,  Polderman and Willems \cite{PW-98}, Martin, Murray and Rouchon \cite{martin-murray-rouchon-2001},  Fabia\'nska and Quadrat \cite{fabianska-quadrat-2007}, Antritter and  Levine \cite{antritter-levine-2010}, Imae, Akasawa, and Kobayashi \cite{imae-akasawa-kobayashi-2015}). Another interesting application of algebraic frames can be found in the paper \cite{EGB} by  Elkadi, Galligo and  Ba, devoted  to the following problem: given a vector of polynomials with gcd 1, find small degree perturbations so that the perturbed polynomials have a large-degree gcd. As discussed in Example~3 of  \cite{EGB}, the perturbations produced by the algorithm presented in this paper do not always have minimal degrees. It would be worthwhile to study if the usage of degree-optimal moving frames  can decrease the degrees of the perturbations. 

 Obviously,  the first column of an algebraic  moving frame $P$ at $\aa$ is a  \emph{B\'ezout vector} of $\mathbf{a}$; that is, a vector comprised of the coefficients appearing in the output of the extended Euclidean algorithm.  In Proposition~\ref{prop-mfbasis}, we prove that  the last $n-1$ columns of $\mf$ comprise  a point-wise linearly independent  basis of the syzygy module of $\mathbf{a}$.
  In Theorem~\ref{thm-bbomf}, we show  that a matrix $\mf$ is a degree-optimal moving frame  at $\aa$  if and only if the first column of $\mf$ is   a B\'ezout vector
of $\mathbf{a}$ of \emph{minimal degree}, and the last $n-1$
columns form a basis of the syzygy module of $\mathbf{a}$ of \emph{optimal
degree}, called a $\mu$-basis \cite{cox-sederberg-chen-1998}. The concept of $\mu$-bases, along with several related concepts such as moving lines and moving curves, have a long history of applications in geometric modeling, originating with works by Sederberg and Chen \cite{sederberg-chen-1995},   Cox,  Sederberg and Chen \cite{cox-sederberg-chen-1998}. Further development of this topic appeared  in  \cite{chen-cox-liu-2005, song-goldman-2009, jia-goldman-2009,tesemma-wang-2014}.

One may attempt to construct an optimal moving frame by putting together a minimal-degree \bez vector and a $\mu$-basis.   {\Blue Indeed,  algorithms for computing $\mu$-bases are well-developed. The most straightforward (and computationally inefficient) approach consists of computing the reduced Gr\"obner basis of the syzygy module with respect to a term-over-position monomial ordering. More efficient algorithms have been developed  by  Cox, Sederberg, and Chen\cite{cox-sederberg-chen-1998},  Zheng and Sederberg \cite{zheng-sederberg-2001}, Chen and Wang  \cite{chen-wang-2002} for the $n=3$ case, and  by Song and Goldman  \cite{song-goldman-2009} and Hong, Hough and Kogan \cite{hhk2017} for arbitrary $n$.  The problem of computing a $\mu$-basis also can be viewed as a particular case of the problem of computing optimal-degree kernels of $m\times n$ polynomial matrices of rank $m$ (see for instance Beelen \cite{beelen1987},   Antoniou,  Vardulakis, and Vologiannidis \cite{antoniou2005},  Zhou, Labahn, and Storjohann \cite{zhou-2012} and references therein).
On the contrary,  our literature search did not yield any articles devoted to the problem of finding an efficient algorithm for computing a minimal-degree \bez vector. Of course, one can compute such a vector by a brute-force  method, namely by searching for a \bez vector of        
a fixed degree, starting  from degree zero, increasing the degree by one, and terminating the search once a \bez vector is  found, but this procedure is  very inefficient.

Alternatively, one can first construct a non-optimal moving frame by algorithms using, for instance, a generalized version of Euclid's extended gcd algorithm, as described by Polderman and Willems in \cite{PW-98}, or various algorithms presented in the literature devoted to the constructive Quillen-Suslin theorem and the related problem of unimodular completion: Fitchas and Galligo \cite{fitchas-galligo-1990}, Logar and Sturmfels \cite{logar-sturmfels-1992}, Caniglia,  Corti\~nas, Dan\'on,  Heintz, Krick, and  Solern\'o \cite{caniglia-1993}, Park and Woodburn \cite{park-woodburn-1995}, Lombardi and Yengui \cite{lombardi-yengui-2005}, Fabia\'nska and Quadrat \cite{fabianska-quadrat-2007},  Zhou-Labahn
  \cite{zhou-labahn-2014}.  Then a degree-reduction procedure can be performed, for instance, by computing the Popov normal form of the last $n-1$ columns of a non-optimal moving frame, as discussed in \cite{beckermann-labahn-villard-2006}, and then reducing the degree of its first column.  We discuss this approach in Section~\ref{sect-compare}, and demonstrate that it is less efficient  than the direct algorithm presented here.  

The advantage of the algorithm presented here is that it simultaneously constructs a minimal-degree \bez vector and a $\mu$-basis.}  Theorem~\ref{thm-bez},  proved in this paper, is crucial for our algorithm, because it shows how a minimal-degree \bez vector can be read off a   Sylvester-type matrix  associated with $\aa$,  the same matrix that has been used in  \cite{hhk2017} for computing a $\mu$-basis. 
 This theorem leads to an algorithm consisting of the following three steps: (1)
 build a  Sylvester-type  $(2d+1)\times(nd+n)$ matrix $A$, associated with $\aa$, where $d$ is the maximal degree of the components of the  vector $\aa$,  and  append an additional column to $A$;  (2) run a single partial row-echelon reduction of the resulting  $(2d+1)\times(nd+n+1)$ matrix; (3) read off an optimal moving frame from appropriate columns of the partial reduced row-echelon form. We implemented the algorithm  in the computer algebra system Maple. The codes and examples are available on the web:
\url{http://www.math.ncsu.edu/~zchough/frame.html}. The algorithm presented here has a natural generalization to unimodular matrix inputs $\bf A$. In the matrix case, partial row echelon reduction is performed on the matrix obtained by stacking together Sylvester-type matrices corresponding to each row of $\bf A$. The details will appear in the dissertation \cite{hough-2018} of the second author.


Along with  the developing a new algorithm for computing an optimal moving frame, we prove new results about the degrees of optimal moving frames and its building blocks.  These degrees play an important role in the classification of rational curves, because although a degree-optimal moving frame is not unique,  its columns have canonical degrees.  The list of degrees of the last $n-1$ columns ($\mu$-basis columns)  is called  the $\mu$-type of an input polynomial vector,  and $\mu$-strata analysis  was performed in D'Andrea \cite{andrea-2004}, Cox and Iarrobino \cite{CI-2015}.  In Theorem~\ref{bez-mu}, we  show  that the degree of the first column (\bez vector) is bounded by the maximal degree of the other columns, while Proposition~\ref{mu-type} shows that this is the only restriction that  the $\mu$-type imposes on  the degree of a minimal  \bez vector. Thus, one can  refine the $\mu$-strata analysis to the $(\beta,\mu)$-strata analysis, where $\beta$  denotes the degree of a minimal-degree \bez vector. This  work can have potential applications  to rational curve classification problems.  In Proposition~\ref{deg} and Theorem~\ref{generic-deg}, we establish sharp lower and upper bounds for the degree of an optimal moving frame and show that  for a generic vector $\aa$, the degree of an optimal moving frame equals to the sharp  lower bound. 

The majority of frames  in differential geometry have a group-equivariance property.  For a curve in the three dimensional space, the Frenet frame is a classical example of a Euclidean group-equivariant frame. However, alternative geometric frames, in particular rotation minimizing frames, appear in applications in computer aided  geometric design, geometric modeling, and computer graphics (see, for  instance, \cite{guggenheimer-1989}, \cite{wang-et-al-2008}, \cite{farouki-et-al-2014},  \cite {farouki-2016}  and references therein). A method for deriving equivariant moving frames  for higher-dimensional objects  and for non-Euclidean geometries has been developed by Cartan (such as in \cite{C35}), who used moving frames  to solve various group-equivalence problems (see \cite{Gug63}, \cite{ivey-landsberg-2016}, \cite{clelland-2017} for modern introduction into Cartan's approach).  The moving frame method was further developed and generalized  by Griffiths \cite{Gr74}, Green \cite{Green78}, Fels and Olver \cite{FO99}, and many others.   Group-equivariant moving frames have  a wide range of applications  to problems in  mathematics, science, and engineering (see  \cite{olver-150} for an overview).
    In Section~\ref{sect-geom-mf},we show that a simple modification of any deterministic algorithm for producing a degree-optimal algebraic moving frame leads to an algorithm that produces a  $GL_n(\k)$-equivariant   degree-optimal moving frame.
This opens the possibility of exploiting   a combination of  important geometric and algebraic  properties to address  equivalence and symmetry problems.  
 
\emph{We  now summarize  each of the following sections emphasizing  the new results  therein contained.} In Section~\ref{sect-mf-bv-syz}, we give precise definitions of a  degree-optimal moving frame, a minimal-degree B\'ezout  vector, and a $\mu$-basis. We show the relationships between these objects. In particular, Theorem~\ref{thm-bbomf}  states  that a minimal-degree B\'ezout  vector and a $\mu$-basis are the building blocks of any degree-optimal moving frame.
This result, although essential to our study,   is by no means  surprising and is easily deducible from known results. Theorem \ref{bez-mu}  and Proposition \ref{mu-type} establish important  relationships
between the degrees of a $\mu$-basis  and the degree of a minimal \bez vector.  
 In Section~\ref{sect-la}, by introducing a modified Sylvester-type matrix $A$, associated with an input vector $\aa$, we reduce the problem of constructing a degree-optimal moving frame to a linear algebra problem over $\k$. Theorems~\ref{thm-bez} and~\ref{thm-mu} show how a minimal-degree  B\'ezout  vector and a $\mu$-basis, respectively, can be constructed from the matrix $A$.  Theorem~\ref{thm-bez} is new, while Theorem~\ref{thm-mu} is a slight modification of Theorem~27 in
\cite{hhk2017}.
In Section~\ref{sect-degree}, we prove new results about the degree of an optimal moving frame. In particular, in Proposition~\ref{deg}, we establish  the sharp lower  bound $\left\lceil\frac{ d}{n-1}\right\rceil$ and the sharp upper bound $d$ for the degree of an optimal moving frame, and in Theorem~\ref{generic-deg}, we prove  that for a generic vector $\aa$, the degree of every degree-optimal moving frame at $\aa$ equals to the sharp lower bound.
In Section~\ref{sect-alg}, we present a degree-optimal moving frame (OMF) algorithm.  The algorithm  exploits the fact that the construction procedures for a minimal-degree  B\'ezout  vector and for a $\mu$-basis, suggested by Theorems~\ref{thm-bez} and~\ref{thm-mu}, can be accomplished simultaneously by a single  partial row-echelon reduction   of  a  $(2d+1)\times(nd+n+1)$ matrix over $\k$. In Proposition~\ref{prop-complexity}, we prove that the theoretical (worst-case asymptotic) complexity of the OMF algorithm equals to $O(d^{2}n + d^{3} + n^{2})$, and we trace the algorithm on our running example.
 In Section~\ref{sect-compare}, we compare our algorithm with other possible approaches.
 In Section~\ref{sect-geom-mf}, we show that  important algebraic properties of the frames produced by the OMF algorithm can be enhanced by a group-equivariant property which plays a crucial role in geometric  moving frame theory.


\section{Moving frames, {B\'ezout} vectors, and syzygies}

\label{sect-mf-bv-syz} In this section, we give the definitions of moving
frame and degree-optimal moving frame,  and explore the relationships between moving frames, syzygies,
and {B\'ezout} vectors.


\subsection{Basis definitions and notation}

\label{ssect-def} Throughout the paper, $\mathbb{K}$ is an arbitrary field,
$\overline{\mathbb{K}}$ is its algebraic closure, and $\mathbb{K}[s]$ is the
ring of {\Red univariate} polynomials over $\mathbb{K}$.  For arbitrary natural numbers $t$ and
$m$, by $\mathbb{K}[s]^{t\times m}$ we denote the set of $t\times m$ matrices
with polynomial entries.  The set of $n\times n$ invertible matrices over
$\mathbb{K}[s]$ is denoted as $GL_{n}(\mathbb{K}[s])$. It is well-known and
easy to show that the determinant of such matrices is a nonzero element of
$\mathbb{K}$.
For a matrix $N$, we will use notation $N_{*i}$ to denote its $i$-th column.
For a square matrix, $|N|$ denotes its determinant.

By $\mathbb{K}[s]^{m}$ we denote the set of vectors of length $m$ with
polynomial entries. All vectors are implicitly assumed to be \emph{column
vectors}, unless specifically stated otherwise. Superscript $\phantom{.}^{T}$
denotes transposition. We will use the following definitions of the degree and
leading vector of a polynomial vector:

\begin{definition}
[Degree and Leading Vector]\label{def-basic} For $\mathbf{h}=[h_{1}%
,\dots,h_{m}]\in\mathbb{K}[s]^{m}$ we define the \emph{degree} and the
\emph{leading vector} of $\hh$ as follows:

\begin{itemize}
\item $\displaystyle{\deg(\mathbf{h})=\max_{i=1,\dots,m}\deg(h_{i})}$.

\item $LV(\mathbf{h})=[\mathrm{coeff}(h_{1},t),\dots,\mathrm{coeff}%
(h_{m},t)]^{T}\in\mathbb{K}^{n}$, where $t=\deg(\mathbf{h})$ and
$\mathrm{coeff}(h_{i},t)$ denotes the coefficient of $s^{t}$ in $h_{i}$.

\item We will say that a set of polynomial vectors $\mathbf{h}_{1}%
,\dots,\mathbf{h}_{k}$ is \emph{degree-ordered} if $\deg(\mathbf{h}_{1}%
)\leq\dotsm\leq\deg(\mathbf{h}_{k})$
\end{itemize}
\end{definition}

\begin{example}
Let $\mathbf{h}=\left[
\begin{array}
[c]{c}%
9-12s-s^2\\
8+15s\\
-7-5s+s^2
\end{array}
\right]  $. Then $\deg(\mathbf{h})=2$ and $LV(\mathbf{h})=\left[
\begin{array}
[c]{r}%
-1\\
0\\
1
\end{array}
\right]  .$
\end{example}

By $\mathbb{K}[s]^{m}_{t}$ we denote the set of vectors of length $m$ of
degree at most $t$. 

\emph{Throughout the paper, $\mathbf{a}\in\mathbb{K}[s]^{n}$ is
assumed to be a nonzero row vector with $n>1$.}


\subsection{Algebraic moving frames and degree optimality}

\label{ssect-mf}

\begin{definition}
[Algebraic Moving Frame]\label{def-mf} For a given nonzero \emph{row} vector
$\mathbf{a}\in\mathbb{K}[s]^{n}$, with $n>1$, an \emph{(algebraic) moving frame at
$\mathbf{a}$} is a matrix $P\in GL_{n}(\mathbb{K}[s])$, such that
\begin{equation}
\label{eq-mf}\mathbf{a}\, P=[\gcd(\mathbf{a}),0,\dots,0],
\end{equation}
where $\gcd(\mathbf{a})$ denotes the greatest monic common devisor of
$\mathbf{a}$.
\end{definition}

{\Red We clarify that by a zero polynomial we mean a polynomial with all its coefficients equal to zero (recall that when $\k$ is a finite field, there may exist a polynomial with nonzero coefficients, which nonetheless is a zero function on $\k$).} As we will show below, a moving frame at $\mathbf{a}$ always exists and is
not unique. For instance, if $P$ is a moving frame at $\mathbf{a}$, then a
matrix obtained from $P$ by permuting the last $n-1$ columns of $P$ is also a
moving frame at $\mathbf{a}$.  The set of all moving frames at $\mathbf{a}$
will be denoted ${\operatorname{mf}}(\mathbf{a})$.  
We are interested in
constructing a moving frame of optimal degree.

\begin{definition}
[Degree-Optimal Algebraic Moving Frame]\label{def-omf}\emph{A moving frame} $P$ at $\mathbf{a}$
is called \emph{degree-optimal} if 

\begin{enumerate}

\item $\deg(P_{*2})\leq\cdots\leq\deg(P_{*n})$, 

\item if $P^{\prime}$ is another moving frame at $\mathbf{a}$, such that
$\deg(P^{\prime}_{*2})\leq\cdots\leq\deg(P^{\prime}_{*n})$, then
\[
\deg(P_{*i})\leq\deg(P^{\prime}_{*i})\quad\text{ for }\quad i=1,\dots, n.
\]

\end{enumerate}
\Red{In other words, we require that the last $n-1$ columns of $P$  (which are interchangeable) are  degree-ordered, and that \emph{all} columns of $P$ are degree-optimal.}
\end{definition}

For simplicity, we will often use the term \emph{optimal moving frame} or  \emph{degree-optimal frame} instead of  \emph{degree-optimal algebraic moving frame}. A degree-optimal moving frame also is not unique, but it is clear from the definition that all optimal moving frames at $\aa$ have the same column-wise degrees. 




\begin{example}
[Running Example]\label{ex-prob} We will show  that 
\small{$P=\left[
\begin{array}
[c]{ccc}%
2-s & 3-3\,s-{s}^{2} & 9-12\,s-{s}^{2}\\
1+2\,s & 2+5\,s+{s}^{2} & 8+15\,s\\
-1-s & -2-2\,s & -7-5\,s+{s}^{2}%
\end{array}
\right] $} is an optimal-degree frame at $\aa=\left[2+s+s^{4} \quad  3+s^{2}+s^{4} \quad 6+2s^3+s^{4}\right]$.
\end{example}

One can immediately notice that the moving frame is closely related to the
{B\'ezout} identity and to syzygies of $\mathbf{a}$. We explore and exploit
this relationship in the following  subsections.


\subsection{{B\'ezout} vectors}


\begin{definition}
[{B\'ezout} Vector]\label{def-bez}\textrm{A \emph{{B\'ezout} vector} of a
row vector $\mathbf{a}\in\mathbb{K}[s]^{n}$ is a column vector
$\mathbf{h}=[h_{1},\dots, h_{n}]^{T}\in\mathbb{K}[s]^{n}$, such that
\[
\mathbf{a }\, \mathbf{h}=\gcd(\mathbf{a}).
\]
}
\end{definition}

The set of all {B\'ezout} vectors of $\mathbf{a}$ is denoted by
$\operatorname{Bez}(\mathbf{a})$ and the set of {B\'ezout} vectors of degree
at most $d$ is denoted $\operatorname{Bez}_{d}(\mathbf{a})$.


\begin{definition}
[Minimal {{B\'ezout} Vector}]\label{def-bezm} A \emph{{B\'ezout} vector}
$\mathbf{h}$ of $\mathbf{a}=[a_{1},\dots, a_{n}]\in\mathbb{K}[s]^{n}$ is said
to be \emph{of minimal degree} if
\[
\deg(\mathbf{h})=\min_{\mathbf{h}^{\prime}\in\operatorname{Bez}(\mathbf{a})}
\deg(\hh^{\prime}).
\]

\end{definition}

The existence of a {B\'ezout} vector can be proven using the extended Euclidean
algorithm. Moreover, since the set of the degrees of all {B\'ezout} vectors
is well-ordered, there is a minimal-degree {B\'ezout} vector.  It is clear that the first column of a moving frame $P$ at $\mathbf{a}$ is a \bez vector of $\aa$, and therefore, in this paper,
we provide, in particular, a simple linear algebra algorithm to construct a {B\'ezout} vector
of minimal degree.

\subsection{Syzygies and $\mu$-bases}


\begin{definition}
[Syzygy]\label{def-syz} A \emph{syzygy} of a nonzero row vector
$\mathbf{a}=[a_{1},\dots, a_{n}]\in\mathbb{K}[s]^{n}$, for $n>1$, is a column
vector $\mathbf{h}\in\mathbb{K}[s]^{n}$, such that
\[
\mathbf{a }\, \mathbf{h}=0.
\]

\end{definition}

The set of all syzygies of $\mathbf{a}$ is denoted by $\operatorname{syz}%
(\mathbf{a})$, and the set of syzygies of degree at most $d$ is denoted
$\operatorname{syz}_{d}(\mathbf{a})$. It is easy to see that
$\operatorname{syz}(\mathbf{a})$ is a module. The next proposition shows that
the last $n-1$ columns of a moving frame form a basis of $\operatorname{syz}%
(\mathbf{a})$.

\begin{proposition}
[Basis of Syzygies]\label{prop-mfbasis} Let $P\in{\operatorname{mf}%
}(\mathbf{a})$. Then the columns $P_{*2},\dots, P_{*n}$ form a basis of
$\operatorname{syz}(\mathbf{a})$.
\end{proposition}


\begin{proof}
We need to show that $P_{*2},\dots, P_{*n}$ generate $\operatorname{syz}%
(\mathbf{a})$ and that they are linearly independent over $\mathbb{K}[s]$.

\begin{enumerate}
\item From \eqref{eq-mf}, it follows that $\mathbf{a}\,P_{*2}=\cdots
=\mathbf{a}\,P_{*n}=0$. Therefore, $P_{*2},\dots, P_{*n}\in\operatorname{syz}%
(\mathbf{a})$. It remains to show that an arbitrary $\mathbf{h}\in
\operatorname{syz}(\mathbf{a})$ can be expressed as a linear combination of
$P_{*2},\dots, P_{*n}\in\operatorname{syz}(\mathbf{a})$ over $\mathbb{K}[s]$.
Trivially we have
\begin{equation}
\mathbf{h}=P(P^{-1}\hh). \label{trivial}%
\end{equation}
From \eqref{eq-mf}, it follows that 
$\mathbf{a}=
\left[
\begin{array}[c]{cccc}%
\gcd(\mathbf{a}) & 0 & \cdots & 0
\end{array}
\right] 
P^{-1}$ and, 
therefore, the first row of $P^{-1}$ is the vector
$\tilde{\mathbf{a}}=\mathbf{a}/\gcd(\mathbf{a})$. 

Hence, since
$\mathbf{a}\,\mathbf{h}=0$, then $P^{-1}\mathbf{h}=[0,g_{2}(s),\dots
,g_{n}(s)]^{T}$ for some $g_{i}(s)\in\mathbb{K}[s]$, $i=2,\dots,n$. Then
(\ref{trivial}) implies:
\[
\mathbf{h}=\sum_{i=2}^{n}\,g_{i}P_{*i}.
\]
Thus $P_{*2},\dots, P_{*n}$ generate $\operatorname{syz}(\mathbf{a})$.

\item Let $f_{2},\ldots,f_{n}\in\mathbb{K}[s]$ be such that
\begin{equation}
\label{eq-dep}f_{2}\, P_{*2}+\cdots+f_{n}\,P_{*n}=0.
\end{equation}
Then 
\(P \tilde f = 0,  \text{ where } \tilde f = [0, f_2, . . . , f_n]^T,\)
and, since $P$ is invertible, it follows that $f_{2}=\cdots=f_{n}=0$. 
\end{enumerate}
\end{proof}


\begin{remark}
\textrm{Note that the proof of Proposition~\ref{prop-mfbasis} is valid over
the ring of polynomials in several variables. {\Red Thus, if a moving frame exists  in  the multivariable case, it follows that its last $n-1$ columns comprise  a basis
of $\syz(\aa)$.} 
 It is well-known that in the
multivariable case there exists $\mathbf{a}$ for which $\operatorname{syz}%
(\mathbf{a})$ is not free and then, from Proposition~\ref{prop-mfbasis}, it
immediately follows that a moving frame at $\mathbf{a}$ does not exist. }
\end{remark}


In the univariate case, both the existence of an algebraic moving frames and freeness of the syzygy module are well-known.
We do not, however, use these results, but as  a by-product of developing an algorithm for constructing an optimal-degree moving frame, we produce a self-contained elementary linear algebra proof of their existence.


\begin{definition}
[$\mu$-basis]\label{def-mb} For a nonzero row vector $\aa\in\mathbb{K}[s]^{n}$,
a set of $n-1$ polynomial vectors $\mathbf{u}_{1},\dots, \mathbf{u}_{n-1 }%
\in\mathbb{K}[s]^{n}$ is called a \emph{$\mu$-basis of $\mathbf{a}$}, or,
equivalently, a \emph{$\mu$-basis of $\operatorname{syz}(\mathbf{a})$}, if the
following two properties hold:

\begin{enumerate}
\item $LV(\mathbf{u}_{1}),\dots,LV(\mathbf{u}_{n-1})$ are linearly independent over
$\mathbb{K}$;

\item $\mathbf{u}_{1},\dots, \mathbf{u}_{n-1 }$ generate $\operatorname{syz}%
(\mathbf{a})$, the syzygy module of $\mathbf{a}$.
\end{enumerate}
\end{definition}

A $\mu$-basis is, indeed, a \emph{basis} of $\operatorname{syz}(\mathbf{a})$
as justified by the following:

\begin{lemma}
\label{lem-basis} Let polynomial vectors $\mathbf{h}_{1},\dots, \mathbf{h}%
_{l}\in\mathbb{K}[s]^{n}$ be such that $LV(h_{1}),\dots,LV( h_{l})$ are linearly
independent over $\mathbb{K}$. Then $\mathbf{h}_{1},\dots, \mathbf{h}_{l}$ are linearly
independent over $\mathbb{K}[s]$.
\end{lemma}

\begin{proof}
Assume that $\mathbf{h}_{1},\dots,\mathbf{h}_{l}$ are linearly
\emph{dependent} over $\mathbb{K}[s]$, i.e.~there exist polynomials
$f_{1},\dots,f_{l}\in\mathbb{K}[s]$, not all zero, such that
\begin{equation}
\sum_{i=1}^{l}f_{i}\,\mathbf{h}_{i}=0.\label{eq-zero}%
\end{equation}
Let $t=\displaystyle{\max_{i=1,\dots,l}}\, \deg(f_{i}\,\mathbf{h}_{i}) $ and
let $\mathcal{I}$ be the set of indices on which this maximum is achieved.
Then \eqref{eq-zero} implies
\[
\sum_{i\in\mathcal{I}}LC(f_{i})\,LV(\mathbf{h}_{i})=0,
\]
where $LC(f_{i})$ is the leading coefficient of $f_{i}$ and is nonzero for
$i\in\mathcal{I}$. This identity contradicts our assumption that
$LV(\mathbf{h}_{1}),\dots,LV(\mathbf{h}_{l})$ are linearly independent over
$\mathbb{K}$.
\end{proof}

In \cite{cox-sederberg-chen-1998}, Hilbert polynomials and the Hilbert Syzygy Theorem were used to show the existence of a basis of $\operatorname{syz}(\aa)$ with
especially nice properties, called a $\mu$-basis. An alternative proof of the
existence of a $\mu$-basis based on elementary linear algebra was given in
\cite{hhk2017}.

In Propositions~\ref{prop-equiv-def} below, we list some 
properties of $\mu$-bases, which are equivalent to its definition. The proof can be easily  adapted from  Theorems~1 and~2 in
\cite{song-goldman-2009} and is omitted here. Only the properties used in the current paper are listed. For a more
comprehensive list of properties of a $\mu$-basis
see 
\cite{song-goldman-2009}.

\begin{proposition}
[Equivalent properties]\label{prop-equiv-def} Let $\mathbf{u}_{1},\dots,
\mathbf{u}_{n-1}$ be a degree-ordered basis of $\operatorname{syz}(\mathbf{a})$,
i.e. $\deg(\mathbf{u}_{1})\leq\cdots\leq\deg(\mathbf{u}_{n-1})$. Then the
following statements are equivalent: 

\begin{enumerate}

\item \label{pr-LV} \emph{[independence of the leading vectors]}
$\mathbf{u}_{1},\dots, \mathbf{u}_{n-1}$ is a $\mu$-basis.

\item \label{pr-reduced}\emph{[reduced representation]} For every
$\mathbf{h}\in\operatorname{syz}(\aa)$, there exist  polynomials $f_{1},\dots, f_{n-1}%
$ such that $\deg(f_{i}\,\mathbf{u}_{i}) \leq \deg(\mathbf{h})$
and
\begin{equation}
\label{eq-reduced}\mathbf{h}=\sum_{i=1}^{n-1} f_{i}\,\mathbf{u}_{i}.
\end{equation}

\item \label{pr-min} \emph{[optimality of the degrees]} If $\mathbf{h}_{1},
\dots, \mathbf{h}_{n-1}$ is another basis of $\operatorname{syz}(\mathbf{a})$,
such that $\deg(\mathbf{h}_{1})\leq\dots\leq\deg(\mathbf{h}_{n-1})$, then
$\deg(\mathbf{u}_{i})\leq\deg(\mathbf{h}_{i})$ for $i=1,\dots, n-1$.
\end{enumerate}
\end{proposition}

We proceed with proving point-wise linear  independence of  the vectors in a $\mu$-basis.  In Theorem~1 of \cite{song-goldman-2009}, $\mu$-bases of real polynomial vectors were considered,  and  point-wise independence of the vectors in a $\mu$-basis was proven for every $s$ in $\mathbb R$. This proof can be word-by-word adapted to $\mu$-bases of polynomial vectors over $\k$ to show  point-wise independence of vectors in a $\mu$-basis for every $s$ in $\k$.   To prove Theorem~\ref{thm-bbomf} of our paper, however, we need a slightly stronger result:   point-wise independence of the vectors in a $\mu$-basis for every $s$ in  $\kk$. To arrive at this result, we first prove the following lemma. In this lemma and the following proposition, we use  
$\syz_{\k[s]}(\aa)$ to denote the syzygy module of $\aa$ over the polynomial ring $\k[s]$, and $\syz_{\kk[s]}(\aa)$ to denote  the syzygy module of $\aa$ over the polynomial ring $\kk[s]$. Elsewhere, we use a shorter notation 
$\syz(\aa)=\syz_{\k[s]}(\aa)$. 

\begin{lemma}\label{lem-ext} 
If $\uu_1,\dots,\uu_{n-1}$ is a $\mu$-basis of 
$\syz_{\k[s]}(\aa)$, then  $\uu_1,\dots,\uu_{n-1}$ is a $\mu$-basis of 
$\syz_{\kk[s]}(\aa)$. 
\end{lemma}
\begin{proof}Since $LV(\uu_1),\dots,LV(\uu_{n-1})$ are independent over $\k$, they also are independent over $\kk$. Thus, it remains to show that $\uu_1,\dots,\uu_{n-1}$ generate $\syz_{\kk[s]}(\aa)$. For an arbitrary $\hh=[h_1,\dots,h_n]^T \in  \syz_{\kk[s]}(\aa)$, consider the field  extension $\mathbb H$ of $\k$ generated by all the coefficients of the polynomials  $h_1,\dots,h_n$. Then $\mathbb H$ is a finite algebraic extension of $\k$ and, therefore, by one of the standard theorems of field theory (see, for example, the first two theorems in Section 41 of \cite{VDW}), $\mathbb H$ is a finite-dimensional vector space over $\k$.  Let $\gamma_1,\dots,\gamma_r\in \mathbb H\subset\kk$ be a vector space basis of  $\mathbb H$ over $\k$. By expanding each of the coefficients in $\hh$ in this basis, we can write $\hh$ as
\beq\label{eq-h1} \hh=\gamma_1\ww_1+\dots+\gamma_r\ww_r ,\eeq
for some $\ww_1,\dots,\ww_r\in\k[s]^n$. Multiplying by $\aa$ on the left, we get
\beq\label{eq-h2} 0=\gamma_1\aa\, \ww_1+\dots+\gamma_r\aa\, \ww_r .\eeq
Assume there exists $i\in\{1,\dots,r\}$ such that $\aa\, \ww_i\neq 0$. Let $k=\deg(\aa\, \ww_i)$ and let $b_j\in\k$ be the coefficient of the monomial $s^k$ in the polynomial
$\aa\, \ww_j$ for $j=1,\dots,r.$ Then, from \eqref{eq-h2}, we have
\beq\label{eq-h3} \nonumber 0=\gamma_1b_1+\dots+\gamma_r\,b_r .\eeq
Since $b_i\neq0$, this contradicts the assumption that $\gamma_1,\dots,\gamma_k$ is a vector space basis of  $\mathbb H$ over $\k$.
Thus, it must be the case that 
$$ \aa\ww_i=0 \text{ for all } i=1,\dots,r$$
and, therefore, \eqref{eq-h1} implies that the module $\syz_{\kk[s]}(\aa)$ is generated by $\syz_{\k[s]}(\aa)$. Since $\syz_{\k[s]}(\aa)$ is generated by $\uu_1,\dots,\uu_{n-1}$, this completes the proof.
\end{proof}

\begin{proposition}
[Point-wise independence over $\kk$]\label{prop-pli} If $\mathbf{u}_{1},\dots,
\mathbf{u}_{n-1}$ is a $\mu$-basis of $\syz_{\k[s]}(\aa)$, then for any value
$s\in\overline{\mathbb{K}}$,  the vectors $\mathbf{u}_{1}(s),\dots,
\mathbf{u}_{n-1}(s)$ are linearly independent over $\overline{\mathbb{K}}$. 
\end{proposition}


\begin{proof}
Suppose there exists $s_{0}\in\overline{\mathbb{K}}$ such that $\uu_{1}%
(s_{0}),\dots, \uu_{n-1}(s_{0})$ are linearly dependent over $\overline{
\mathbb{K}}$. Then there exist constants $\alpha_{1},\dots,
\alpha_{n-1}\in\overline{\mathbb{K}}$, not all zero, such that
\[
\alpha_{1}\,\mathbf{u}_{1}(s_{0})+\dots+\alpha_{n-1}\,\mathbf{u}_{n-1}%
(s_{0})=0.
\]
Let $i=\max\{j\,|\, \alpha_{j}\neq0\}$ and let
\[
\mathbf{h}= \alpha_{1}\,\mathbf{u}_{1}+\dots+\alpha_{i}\,\mathbf{u}_{i}.
\]
Then $\mathbf{h}\in\operatorname{syz}_{\kk[s]}(\mathbf{a})$ and is not identically
zero, but $\mathbf{h}(s_{0})=0$. It follows that $\gcd(\mathbf{h})\neq1$ in
$\kk[s]$ and, therefore, $\tilde \hh=\frac1 {\gcd(\mathbf{h})}\,
\mathbf{h}$  belongs to  $\syz_{\kk[s]}(\aa)$ and has degree strictly less than the degree of $\mathbf{h}$.

By Lemma~\ref{lem-ext},  $\mathbf{u}_{1},\dots,
\mathbf{u}_{n-1}$ is a $\mu$-basis of  $\syz_{\kk[s]}(\aa)$ and, since   $$\mathbf{u}_{i}=\frac1 {\alpha_{i}}\,\left( \gcd(\mathbf{h}%
)\,\tilde{\mathbf{h}}- \alpha_{1}\,\mathbf{u}_{1}-\dots-\alpha_{i-1}
\,\mathbf{u}_{i-1}\right),$$ the set of syzygies
\[
\{\mathbf{u}_{1}, \dots, \mathbf{u}_{i-1}, \mathbf{u}_{i+1},\dots,
\mathbf{u}_{n-1}, \tilde{\mathbf{h}}\}
\]
is a basis  of  $\syz_{\kk[s]}(\aa)$. Ordering it by degree and observing that  $\deg(\tilde{
\mathbf{h}})<\deg(\mathbf{h})=\deg(\mathbf{u}_{i})$ leads to a contradiction
with the degree optimality property of a $\mu$-basis.
\end{proof}

\subsection{The building blocks of a degree-optimal moving frame}

From the discussions of the last section, it does not come as unexpected that a
{B\'ezout} vector and a set of point-wise independent syzygies can serve as
building blocks for a moving frame.


\begin{proposition}
[Building blocks of a moving frame]\label{prop-bb} For a nonzero
$\mathbf{a}\in\mathbb{K}[s]^{n}$, let $\mathbf{h}_{1},\dots, \mathbf{h}_{n-1}$
be elements of $\operatorname{syz}(\mathbf{a})$ such that, for every
$s\in\overline{\mathbb{K}}$, vectors $\mathbf{h}_{1}(s),\dots, \mathbf{h}%
_{n-1}(s)$ are linearly independent over $\overline{\mathbb{K}}$, and let
$\mathbf{h}_{0}$ be a {B\'ezout} vector of $\aa$. Then the matrix
\[
P=[\mathbf{h}_{0},\mathbf{h}_{1},\dots, \mathbf{h}_{n-1}]
\]
is a moving frame at $\aa$.
\end{proposition}

\begin{proof}
Clearly $\mathbf{a}\,P=[\gcd(\mathbf{a}),0,\dots,0]$. Let $\tilde{
\mathbf{a}}=\frac1 {\gcd(\mathbf{a})}\, \mathbf{a}$, then
\begin{equation}
\label{eq-mfta}\tilde{\mathbf{a}}\,P=[1,0,\dots,0].
\end{equation}
Assume that the determinant of $P$ does not equal to a nonzero constant. Then
there exists $s_{0}\in\overline{\mathbb{K}}$ such that
\(
|\mathbf{h}_{0}(s_{0}),\mathbf{h}_{1}(s_{0}),\dots, \mathbf{h}_{n-1}(s_{0})|=0
\)
and, therefore, there exist constants $\alpha_{0},\dots, \alpha_{n}%
\in\overline{\mathbb{K}}$, not all zero, such that
\[
\alpha_{0}\,\mathbf{h}_{0}(s_{0})+\alpha_{1}\,\mathbf{h}_{1}(s_{0}%
)+\dots+\alpha_{n-1}\, \mathbf{h}_{n-1}(s_{0})=0.
\]
Multiplying on the left by $\tilde{\mathbf{a}}(s_{0})$ and using
\eqref{eq-mfta}, we get $\alpha_{0}=0$.  Then
\[
\alpha_{1}\,\mathbf{h}_{1}(s_{0})+\dots+\alpha_{n-1}\, \mathbf{h}_{n-1}%
(s_{0})=0
\]
for some set of constants $\alpha_{1},\dots, \alpha_{n-1}\in\overline
{\mathbb{K}}$, not all zero. But this contradicts our assumption that for
every $s\in\overline{\mathbb{K}}$, vectors $\hh_{1}(s),\dots, \hh_{n-1}(s)$ are
linearly independent over $\overline{\mathbb{K}}$. Thus, the determinant of $P$
equals to a nonzero constant, and therefore $P$ is a moving frame.
\end{proof}


\begin{theorem}
\label{thm-bbomf} A matrix $P$ is a degree-optimal moving frame at $\mathbf{a}$ if
and only if $P_{*1}$ is a {B\'ezout} vector of $\aa$ of minimal degree and $P_{*2},
\dots,P_{*n-1}$ is a $\mu$-basis of $\mathbf{a}$.
\end{theorem}


\begin{proof}
\hfill

\begin{description}
\item[($\Longrightarrow$)]{\Red Let $P$ be a degree-optimal moving frame at $\mathbf{a}$.
From Definition~\ref{def-omf}, it immediately follows that $P_{*1}$ is a
{B\'ezout} vector of $\aa$ of minimal degree. From Proposition~\ref{prop-mfbasis}, it
follows that $P_{*2}, \dots, P_{*n}$ is a basis of $\operatorname{syz}%
(\mathbf{a})$. Assume $P_{*2}, \dots, P_{*n}$   is not a $\mu$-basis of $\mathbf{a}$, and  let
$\mathbf{u}_{1},\dots, \mathbf{u}_{n-1}$ be a $\mu$-basis.
 From Proposition~\ref{prop-pli}, it follows that the vectors $\mathbf{u}_{1}(s),\dots, \mathbf{u}_{n-1}(s)$ are
independent for all $s\in\overline{\mathbb{K}}$. By Proposition~\ref{prop-bb},
the matrix $P^{\prime}=[P_{*1}, \uu_{1},\dots, \uu_{n-1}]$ is a moving frame at $\aa$. On
the other hand, since  $P_{*2}, \dots, P_{*n}$ is not a $\mu$-basis, then  by Proposition~\ref{prop-equiv-def}, it is not a basis  of optimal degree, and,  therefore, there
 exists $k\in\{1,\dots, n-1\}$, such that $\deg(\mathbf{u}%
_{k})<\deg(P_{*k+1})$.  This
contradicts our assumption that $P$ is degree-optimal. Therefore,
$P_{*2}, \dots, P_{*n}$ is a $\mu$-basis.}

\item ($\Longleftarrow$) Assume $P_{*1}$ is a {B\'ezout} vector of $\aa$ of minimal
degree and $P_{*2}, \dots,P_{*n-1}$ is a $\mu$-basis of $\mathbf{a}$. Then
Proposition~\ref{prop-pli} implies that the vectors $P_{*2}(s), \dots,P_{*n-1}(s)$
are independent for all $s\in\overline{\mathbb{K}}$ and so $P$ is a moving
frame due to Proposition~\ref{prop-bb}. Assume there exists a moving frame
$P^{\prime}$ and an integer $k\in\{1,\dots, n\}$, such that $\deg(P^{\prime
}_{*k})<\deg(P_{*k})$. If $k=1$, then we have a contradiction with the
assumption that $P_{*1}$ is a {B\'ezout} vector of minimal degree. If $k>1$,
we have a contradiction with the degree optimality property of a $\mu$-basis.
Thus $P$ satisfies Definition~~\ref{def-omf} of a degree-optimal moving frame.
\end{description}

\end{proof}

Theorem~\ref{thm-bbomf} implies the following three-step process for
constructing a degree-optimal moving frame at $\mathbf{a}$. 

\begin{enumerate}

\item Construct a {B\'ezout} vector $\mathbf{b}$ of $\mathbf{a}$ of minimal
degree. 

\item Construct a $\mu$-basis $\mathbf{u}_{1},\dots, \mathbf{u}_{n-1}$ of
$\mathbf{a}$. 

\item Let $P=[\mathbf{b},\mathbf{u}_{1},\dots, \mathbf{u}_{n-1}]$. 
\end{enumerate}

However, by exploiting the relationship between these building blocks, we  develop, in Section~\ref{sect-alg}, an algorithm that \emph{simultaneously constructs a {B\'ezout} vector of minimal degree and a $\mu$-basis}, avoiding  redundancies embedded in the above three-step procedure.

\subsection{The $(\beta, \mu)$-type of a polynomial vector}
The degree-optimality  property of a $\mu$-basis, stated in Proposition~\ref{prop-equiv-def}, insures that, although a $\mu$-basis of $\aa$ is not unique,  the  ordered list of the degrees of a $\mu$-basis of $\aa$ is unique. This list is called the $\mu$-type of $\aa$.  Thus the set of polynomial vectors can be split into classes according to their $\mu$-type. An  analysis of the $\mu$-strata of the set of polynomial vectors is given by  D'Andrea \cite{andrea-2004}, Cox and Iarrobino \cite{CI-2015}. Similarly, although  a minimal-degree \bez vector for $\aa$ is not unique, its degree is unique.  If we denote this degree  by $\beta$, we can refine the classification of polynomial vectors by studying their $(\beta,\mu)$-strata. In this section, we explore the relationship between the $\mu$-type and the $\beta$-type of a polynomial vector.%
  
We start by showing that the degree of a minimal-degree {B\'ezout} vector of $\aa$ is bounded  by the maximal degree of a $\mu$-basis of $\aa$. This result is repeatedly used in the paper. 
\begin{theorem}
\label{bez-mu} 
For any nonzero $\aa \in \mathbb{K}[s]^{n}$, and for any minimal-degree {B\'ezout} vector $\bb$  and any $\mu$-basis $\uu_1,\ldots,\uu_{n-1}$ of $\aa$, we have
\begin{enumerate}
\item if $\deg(\aa)=\deg\left(\gcd(\aa)\right)$, then $\deg(\bb)=0$ and $\deg(\uu_i)=0$ for $i=1,\dots,n-1$.
\item otherwise $\deg(\bb) < \max_j\{\deg(\uu_j)\}$.
\end{enumerate} 
\end{theorem}
\begin{proof}\hfill
{\begin{enumerate}
\item The condition $\deg(\aa)=\deg\left(\gcd(\aa)\right)$ implies that $\aa=\gcd(\aa)\,v$, where $v$ is a constant non-zero vector.  In this case, it is obvious how to construct $\bb$ and $\uu_1,\ldots,\uu_{n-1}$, each with constant components.
\item In this case, $\deg(\aa)>\deg\left(\gcd(\aa)\right)$. 
The coefficient of $\aa\,\bb$ for $s^{\deg(\aa)+\deg(\bb)}$ is $LV(\aa)LV(\bb)$. By definition of \bez vector, $\aa\bb = \gcd(\aa)$. Therefore, by our assumption, $\deg(\aa\,\bb)<\deg(\aa)$. Thus $LV(\aa)LV(\bb) = 0$ or, in other words, $LV(\bb) \in LV(\aa)^{\perp}$. Let $\uu_1,\ldots,\uu_{n-1}$ be a $\mu$-basis of $\aa$. 
By a similar argument, since $\aa\,\uu_j = 0$, we have $LV(\uu_j) \in LV(\aa)^{\perp}$  for  $j=1,\dots, n-1$. By definition of a $\mu$-basis,  $LV(\uu_j)$ are linearly independent, and so they form a basis for $LV(\aa)^{\perp}$. Therefore, there exist constants $\alpha_1,\ldots,\alpha_{n-1}$ such that 
\(\displaystyle{
LV(\mathbf{b}) = \sum_{j=1}^{n-1} \alpha_{j}LV(\uu_{j}).}
\)
Suppose that $\deg(\mathbf{b}) \ge \max_{j}
\{\deg(\uu_{j})\}$. Define
\(\displaystyle{
\tilde{\mathbf{b}} = \mathbf{b}- \sum_{j=1}^{n-1} \alpha_{j}\uu_{j}%
s^{\deg(\mathbf{b})-\deg(\uu_{j})}.}
\)
Then $\mathbf{a}\tilde{\mathbf{b}}=\gcd(\aa)$ and $\deg(\tilde{\mathbf{b}}) <
\deg(\mathbf{b})$, a contradiction to the minimality of $\deg(\mathbf{b})$.
Therefore, $\deg(\mathbf{b}) < \max_{j} \{\deg(\uu_{j})\}$.
\end{enumerate}
}
\end{proof}
In the next proposition, we show that, except for the upper bound provided by $\mu_{n-1}-1$, no other additional restrictions on the degree of the minimal \bez vector are imposed by the $\mu$-type, and therefore the $\beta$-type provides an essentially new characteristic of a polynomial vector.

\begin{proposition}\label{mu-type}
Fix $n\ge2$. For all ordered lists of nonnegative integers $\mu_1\le \cdots\le\mu_{n-1}$, with $\mu_{n-1}\not=0$, and for all $j \in \left\{0,\ldots,\mu_{n-1}-1\right\}$, there exists $\aa \in \mathbb{K}[s]^n$ such that {$\gcd(\aa)=1$} and 
\begin{enumerate}
\item for any $\mu$-basis $\uu_1,\ldots,\uu_{n-1}$ of $\aa$, we have $\deg(\uu_i) = \mu_i$, $i=1,\ldots,n-1$.
\item for any  minimal-degree B\'ezout vector $\bb$ of $\aa$, we have $\deg(\bb)=j$.
\end{enumerate}
\end{proposition}
\begin{proof}
{
In the case when $n=2$, given a non-negative integer $\mu_1$ and an integer $j\in \{0,\dots, \mu_1-1\}$, take $\aa = \left[s^{\mu_1-j}, s^{\mu_1}+1\right]$. Then,  obviously $\gcd(\aa)=1$,    vector 
$\mathbf b=[-s^j,1]^T$ is a minimal-degree \bez vector,  and vector  ${\mathbf u}_1=\left[s^{\mu_1}+1, -s^{\mu_1-j}\right]^T$ is the minimal-degree syzygy, which in this case comprises a $\mu$-basis of $\aa$.
Thus $\aa$ has the required properties.

In the case when $n\ge3$,  for the set of integers $\mu_1,\dots, \mu_{n-1},j$ described in the proposition,   take 
$$
 \aa = \left[s^{\mu_{n-1}-j},s^{\mu_{n-1}-j+\mu_1},s^{\mu_{n-1}-j+\mu_1+\mu_2},\ldots,s^{\mu_{n-1}-j+\mu_1+\cdots+\mu_{n-2}},s^{\mu_1+\cdots+\mu_{n-1}}+1\right]. 
$$
Observe that $\gcd(\aa)=1$, and consider the matrix
$$
\mf= \left[\begin{array}{cccccc}
 & s^{\mu_1} & & & 1\\
 & -1 & s^{\mu_2}& &\\
 & & -1 & \ddots &\\
-s^{j} & & & \ddots & s^{\mu_{n-1}}\\
 1 & & & & -s^{\mu_{n-1}-j}
 \end{array}\right].
 $$
It is easy to see that  $\aa P = [1,0,\ldots,0]$ and $|P| = \pm 1$, so $P$ is a moving frame at $\aa$ according to Definition \ref{def-mf}. Therefore, the first column of $P$, i.e vector $\mathbf b=\mf_{*1}$, is a \bez vector of $\aa$, while the remaining columns $\mathbf u_1=P_{*2},\dots,  \mathbf u_{n-1}=P_{*n}$ comprise a basis for the syzygy module of $\aa$   according to Proposition~\ref{prop-mfbasis}.  Clearly $\deg(\bb)=j$, while $\deg \uu_i=\mu_i$ for $i=1,\dots, n-1$.

The leading vectors of $\uu_1,\dots,\uu_{n-1}$ are linearly independent and, therefore,  vectors $\uu_1,\dots,\uu_{n-1}$  comprise  a $\mu$-basis of $\aa$.  To prove that $\bb$ is of minimal degree, suppose, for the sake of contradiction, that there exists a vector $\mathbf{f} = \left[f_1,\ldots,f_n\right]^T \in \mathbb{K}[s]^n$ with $\deg(\mathbf{f})<j$ such that
\beq\label{eq-af}
f_1(s)\,a_1(s)+\ldots+f_n(s)\,a_n(s) = 1 \text{ for all } s.
\eeq
We observe that, since $\mu_{n-1}>0$ and $j<\mu_{n-1}$, then $a_i(0) = 0$ for $i=1,\ldots,n-1$ and $a_n(0)=1$.  Then, by substituting $s=0$ in \eqref{eq-af}, we get  $f_n(0)=1$ and, therefore, $f_n(s)$ is not a zero polynomial.  This implies that   $\deg(f_na_n) = \mu_1+\cdots+\mu_{n-1}+\deg(f_n)$. Therefore, in order for the B\'ezout identity  \eqref{eq-af} to hold, at least one of the remaining $f_ia_i$,  $i=1,\ldots,n-1$, must contain a monomial of degree $\mu_1+\cdots+\mu_{n-1}+\deg(f_n)$ as well. However, we assumed that  $\deg(f_i)<j$ for all $i$, which implies that  $\deg(f_ia_i) < \mu_1+\cdots+\mu_{n-1}$ for $i=1,\ldots,n-1$. Contradiction.
We thus conclude  that $\aa$ has the required properties.
}
\end{proof}

\section{Reduction to a linear algebra problem over $\mathbb{K}$}

\label{sect-la}
\renewcommand\footnotemark{}

In this section, we show that for a vector $\mathbf{a}\in\mathbb{K}[s]^{n}%
_{d}$ such that $\gcd(\mathbf{a})=1$, a {B\'ezout} vector of $\aa$ of minimal degree
and a $\mu$-basis of $\mathbf{a}$ can be obtained from linear relationships
among certain columns of a ${(2d+1)\times(nd+n+1)}$ matrix over $\mathbb{K}$. 
Since essentially the same matrix has been used to construct a $\mu$-basis in \cite{hhk2017}, we later use this result to develop a degree-optimal moving frame algorithm that simultaneously constructs a $\mu$-basis and a minimal {B\'ezout} vector.


\subsection{Sylvester-type matrix $A$ and its properties}

\label{ssect-A}
For a nonzero polynomial row vector
\begin{equation}
\label{eq-a}\mathbf{a }= \displaystyle{\sum_{0\leq i\leq d}[c_{i1}%
,\ldots,c_{in}]s^{i}}%
\end{equation}
of length $n$ and degree $d$, we correspond a $\mathbb{K}^{(2d+1)\times
n(d+1)}$ matrix
\beq
\label{eq-A}A=\left[
\begin{array}
[c]{cccccccccc}%
c_{01} & \cdots & c_{0n} &  &  &  &  &  &  & \\
\vdots & \cdots & \vdots & c_{01} & \cdots & c_{0n} &  &  &  & \\
\vdots & \cdots & \vdots & \vdots & \cdots & \vdots & \ddots &  &  & \\
c_{d1} & \cdots & c_{dn} & \vdots & \cdots & \vdots & \ddots & c_{01} & \cdots
& c_{0n}\\
&  &  & c_{d1} & \cdots & c_{dn} & \ddots & \vdots & \cdots & \vdots\\
&  &  &  &  &  & \ddots & \vdots & \cdots & \vdots\\
&  &  &  &  &  &  & c_{d1} & \cdots & c_{dn}%
\end{array}
\right]
\eeq
with the blank spaces filled by zeros. In other words, matrix $A$ is obtained
by taking $d+1$ copies of a $(d+1)\times n$ block of the coefficients of
polynomials in $\mathbf{a}$. The blocks are repeated horizontally from left to
right, and each block is shifted down by one relative to the previous one.
Matrix $A$ is related to the \emph{generalized resultant matrix} $R$,
appearing on page 333 of \cite{vard-1978}. Indeed, if one takes the
top-left $\mathbb{K}^{2d\times nd}$ submatrix of $A$, transposes this
submatrix, and then permutes certain rows, one obtains $R$. However, the size
and shape of the matrix $A$ turns out to be crucial to our construction.
\begin{example}\label{ex-A}
\textrm{For the row vector $\mathbf{a}$ in the running example
(Example~\ref{ex-prob}), we have $n=3$, $d=4$,
\[
c_{0}=[2,3,6],\, c_{1}=[1,0,0],\,c_{2}=[0,1,0],\,c_{3}=[0,0,2],\,c_{4}=[1,1,1]
\]
and
\[
A=\left[
\begin{array}
[c]{rrrrrrrrrrrrrrr}%
2 & 3 & 6 &  &  &  &  &  &  &  &  &  &  &  & \\
1 & 0 & 0 & 2 & 3 & 6 &  &  &  &  &  &  &  &  & \\
0 & 1 & 0 & 1 & 0 & 0 & 2 & 3 & 6 &  &  &  &  &  & \\
0 & 0 & 2 & 0 & 1 & 0 & 1 & 0 & 0 & 2 & 3 & 6 &  &  & \\
1 & 1 & 1 & 0 & 0 & 2 & 0 & 1 & 0 & 1 & 0 & 0 & 2 & 3 & 6\\
&  &  & 1 & 1 & 1 & 0 & 0 & 2 & 0 & 1 & 0 & 1 & 0 & 0\\
&  &  &  &  &  & 1 & 1 & 1 & 0 & 0 & 2 & 0 & 1 & 0\\
&  &  &  &  &  &  &  &  & 1 & 1 & 1 & 0 & 0 & 2\\
&  &  &  &  &  &  &  &  &  &  &  & 1 & 1 & 1
\end{array}
\right] .
\]}
\end{example}

A visual periodicity of the matrix $A$ is reflected in the periodicity
property of its non-pivotal columns which we are going to precisely define and
exploit below. We remind readers the of the definition of pivotal and non-pivotal columns.

\begin{definition}
A column of any matrix $N$ is called \emph{pivotal} if it is either the first
column and is nonzero or it is linearly independent of all previous columns.
The rest of the columns of $N$ are called \emph{non-pivotal}. The index of a
pivotal (non-pivotal) column is called a \emph{pivotal (non-pivotal) index}.
\end{definition}

From this definition, it follows that every non-pivotal column can be written
as a linear combination of the preceding \emph{pivotal columns}.

We denote the set of pivotal indices of $A$ as $p$ and the set of its
non-pivotal indices as $q$. The following two lemmas, proved in \cite{hhk2017}
(Lemma 17, 19) show how the specific structure of the matrix $A$ is reflected
in the structure of the set of non-pivotal indices~$q$.

\begin{lemma}
[Periodicity]\label{periodic}If $j\in q$ then $j+kn\in q$ for $0\leq
k\leq\left\lfloor \frac{n(d+1)-j}{n}\right\rfloor $. Moreover,
\begin{equation}
\label{eq-periodic}A_{*j}=\sum_{ r<j} \alpha_{r}\, A_{*r}\quad\Longrightarrow
\quad A_{*j+kn}=\sum_{ r<j} \alpha_{r}\, A_{*r+kn},
\end{equation}
where $A_{*j}$ denotes the $j$-th column of $A$.
\end{lemma}


\begin{definition}
\label{def-bnp} Let $q$ be the set of non-pivotal indices. Let $q/(n)$ denote
the set of equivalence classes of $q \text{ modulo } n$. Then the set $\tilde
q=\{ \min\varrho\,| \varrho\in q/(n)\}$ will be called the set of \emph{basic
non-pivotal indices}. The remaining indices in $q$ will be called \emph{periodic non-pivotal indices}.
\end{definition}



\begin{example}
\textrm{\label{ex-period} For the matrix $A$ in Example~\ref{ex-A}, we have
$n=3$ and $q=\{ 8,9,11,12,14,15\}$.  Then $q/(n)=\big\{ \{8,11,14\},\,\{9,12,15\} \}\big\}$ and $\tilde q=\{8,9\}$. }
\end{example}


\begin{lemma}
\label{lem-card} There are exactly $n-1$ basic non-pivotal indices: $|\tilde
q|=n-1$.
\end{lemma}



\subsection{Isomorphism between $\mathbb{K}[s]_{t}^{m}$ and $\mathbb{K}%
^{m(t+1)}$}

\label{ssect-iso}
The second ingredient that we use to reduce our problem to a linear algebra
problem over $\mathbb{K}$ is an explicit isomorphism  between vector spaces
$\mathbb{K}[s]_{t}^{m}$ and $\mathbb{K}^{m(t+1)}$.  Any polynomial $m$-vector
$\mathbf{h}$ of degree at most $t$ can be written as $\mathbf{h}=w_{0}+s
w_{1}+\dots+s^{t} w_{t}$ where $w_{i}=[w_{1i},\dots, w_{mi}]^{T}\in
\mathbb{K}^{m}$. It is clear that the map
\[
\sharp^{m}_{t}\colon\mathbb{K}[s]_{t}^{m}\to\mathbb{K}^{m(t+1)}%
\]
\begin{equation}
\label{iso1}\mathbf{h}\to\mathbf{h}^{\sharp^{m}_{t}}= \left[
\begin{array}
[c]{c}%
w_{0}\\
\vdots\\
w_{t}%
\end{array}
\right]
\end{equation}
is linear. It is easy to check that the inverse of this map
\[
\flat^{m}_{t}\colon\mathbb{K}^{m(t+1)} \to\mathbb{K}[s]_{t}^{m}%
\]
is given by a linear map:
\begin{equation}
\label{iso2}v\to v^{\flat^{m}_{t}} =S^{m}_{t}\, v
\end{equation}
where
\[
S^{m}_{t}=\left[
\begin{array}
[c]{ccc}%
I_{m} & sI_{m} & \cdots s^{t}I_{m}%
\end{array}
\right]  \in\mathbb{K}[s]^{m\times m(t+1)} .
\]
Here $I_{m}$ denotes the $m\times m$ identity matrix. {For the sake of
notational simplicity, we will often write $\sharp$, $\flat$ and $S$ instead
of $\sharp^{m}_{t}$, $\flat^{m}_{t}$ and~$S^{m}_{t}$ when the values of $m$
and $t$ are clear from the context.}


\begin{example}
\label{ex-sharp-flat}\textrm{For $\mathbf{h}\in\mathbb{Q}^{3}_{3}[s]$ given
by
\[
\mathbf{h}=\left[
\begin{array}
[c]{c}%
9-12s-s^{2}\\
8+15s\\
-7-5s+s^2
\end{array}
\right] = \left[
\begin{array}
[c]{r}%
9\\
8\\
-7
\end{array}
\right] +s\,\left[
\begin{array}
[c]{r}%
-12\\
15\\
-5
\end{array}
\right]  + s^{2}\,\left[
\begin{array}
[c]{r}%
-1\\
0\\
1
\end{array}
\right] ,
\]
we have
\[
\mathbf{h}^{\sharp}=[9,\,8,\,-7,\,-12,\,15,\,-5,\,-1,\,0,\,1]^{T}.
\]
Note that
\[
\mathbf{h}=(\mathbf{h}^{\sharp})^{\flat}=S\, \mathbf{h}^{\sharp}=\left[
\begin{array}
[c]{ccc}%
I_{3} & sI_{3} & s^{2}I_{3}
\end{array}
\right]  \mathbf{h}^{\sharp}.
\]
}
\end{example}

With respect to the isomorphisms $\sharp$ and $\flat$, the $\mathbb{K}$-linear
map $\mathbf{a}\colon\mathbb{K}[s]_{d}^{n}\to\mathbb{K}[s]_{2d}$ corresponds
to the $\mathbb{K}$ linear map $A\colon\mathbb{K}^{n(d+1)}\to\mathbb{K}%
^{2d+1}$ in the following sense:

\begin{lemma}
\label{lem-aA}  Let $\mathbf{a}=\displaystyle{\sum_{0\leq j\leq d}c_{j}s^{j}}
\in\mathbb{K}_{d}^{n}[s]$ and $A \in\mathbb{K}^{(2d+1)\times n(d+1)}$ defined
as in \eqref{eq-A}. Then for any $v\in\mathbb{K}^{n(d+1)}$ and any
$\mathbf{h}\in\mathbb{K}[s]^{n}_{d}$:
\begin{equation}
\label{eq-aA}\mathbf{a}\,v^{\flat^{n}_{d}}=(A v)^{\flat^{1}_{2d}} \text{ and }
(\mathbf{a}\, \mathbf{h})^{\sharp^{1}_{2d}}=A \mathbf{h}^{\sharp^{n}_{d}}.
\end{equation}

\end{lemma}

The proof of Lemma~\ref{lem-aA} is straightforward. The proof of the first
equality is explicitly spelled out in \cite{hhk2017} (see Lemma 10). The
second equality follows from the first and the fact that $\flat^{m}_{t}$
and $\sharp^{m}_{t}$ are mutually inverse maps.

\begin{example}
\label{ex-Aa} \textrm{Consider the row vector $\mathbf{a}$ in the running example
(Example~\ref{ex-prob}) and its associated matrix $A$ (Example \ref{ex-A}).
Let $v=[1,2,3,4,5,6,7,8,9,10,11,12,13,14,15]^{T}$. Then
\[
Av= [26,60,98,143,194,57,62,63,42]^{T}%
\]
and so
\[
(Av)^{\flat_{2d}^{1}} = S_{8}^{1}(Av) = 26 + 60s + 98s^{2} + 143s^{3} + 194s^{4}
+ 57s^{5} + 62s^{6} + 63s^{7} + 42s^{8}.
\]
On the other hand, since
\[
v^{\flat^{n}_{d}} = S_{4}^{3}v = \left[
\begin{array}
[c]{c}%
1+4s+7s^{2}+10s^{3}+13s^{4}\\
2+5s+8s^{2}+11s^{3}+14s^{4}\\
3+6s+9s^{2}+12s^{3}+15s^{4}%
\end{array}
\right] ,
\]
we have
\begin{align*}
\mathbf{a }v^{\flat^{n}_{d}}  & = \left[
\begin{array}
[c]{ccc}%
2+s+s^{4} & 3+s^{2}+s^{4} & 6+2s^3+s^{4}%
\end{array}
\right]  \left[
\begin{array}
[c]{c}%
1+4s+7s^{2}+10s^{3}+13s^{4}\\
2+5s+8s^{2}+11s^{3}+14s^{4}\\
3+6s+9s^{2}+12s^{3}+15s^{4}%
\end{array}
\right] \\
& = 42s^{8}+63s^{7}+62s^{6}+57s^{5}+194s^{4}+143s^{3}+98s^{2}+60s+26.
\end{align*}
We observe that
\[
\mathbf{a }v^{\flat^{n}_{d}}=(Av)^{\flat^{1}_{2d}}.
\]
}
\end{example}
We proceed  by showing that if $\gcd(\aa)=1$, then  the matrix $A$ has full rank.  This statement  can be
deduced from  the results about the rank of a different Sylvester-type matrix, $R$, given in Section 2 of
\cite{vard-1978}. We, however,  give a short independent proof using the following lemma, which also will be used in other parts of the paper.
\begin{lemma}
\label{pre-rank} For all $\aa \in \mathbb{K}[s]^{n}$ with
$\gcd(\mathbf{a})=1$ and $\deg(\mathbf{a})=d$ and all $i=0,\ldots,2d$, there exist vectors $\mathbf{h}_i \in\mathbb{K}[s]^{n}$ such that $\deg(\mathbf{h}_i) \le d$ and $\mathbf{a}\,\mathbf{h}_i= s^{i}$.
\end{lemma}

\begin{proof}
Let $\uu_1,\ldots,\uu_{n-1}$ be a $\mu$-basis of $\aa$. We will proceed by induction on $i$.

 \noindent Induction basis: For $i=0$, the statement follows immediately from Theorem \ref{bez-mu} and the well-known fact that $\syz(\aa)$ can be generated by vectors of degree at most $d$ (see, for example, \cite{hhk2017} or \cite{song-goldman-2009}).

\noindent  Induction step: Assume the
statement is true in the $i$-th case i.e. there exists $\mathbf{h}_i%
\in\mathbb{K}[s]^{n}$ with $\deg(\mathbf{h}_i) \le d$ such that $\mathbf{a}\,
\mathbf{h}_i= s^{i}$ ($i \le2d-1$). Then $\mathbf{a }\,(s\mathbf{h}_i)= s^{i+1}$.
Let $\tilde{\mathbf{h}} = s\mathbf{h}_i$. Since $\deg(\mathbf{h}_i) \le d$, it
follows that $\deg(\tilde{\mathbf{h}}) \le d+1$. If $\deg(\tilde{\mathbf{h}})
\le d$, let $\hh_{i+1}=\tilde{\hh}$ and we are done. Otherwise, $\deg(\tilde{\mathbf{h}}) = d+1$. Following a similar argument as in Theorem~\ref{bez-mu}, the coefficient of $\mathbf{a}\tilde{\mathbf{h}}$ for $s^{2d+1}$ is
$LV(\mathbf{a})\,LV(\tilde{\mathbf{h}})$, and since we assumed $i \le2d-1$, it
must be that $LV(\mathbf{a})LV(\tilde{\mathbf{h}}) = 0$. Thus, there exist
constants $\alpha_{1},\ldots,\alpha_{n-1}$ such that
\(
LV(\tilde{\mathbf{h}}) = \sum_{j=1}^{n-1} \alpha_{j}LV(\uu_{j}).
\)
Define
\[
\mathbf{h}_{i+1}= \tilde{\mathbf{h}} - \sum_{j=1}^{n-1} \alpha_{j}%
\uu_{j}s^{d+1-\deg(\uu_{j})}.
\]
Then $\mathbf{a}\,\mathbf{h}_{i+1} = s^{i+1}$ and $\deg(\mathbf{h}_{i+1}) <
\deg(\tilde{\mathbf{h}})$, which means $\deg(\mathbf{h}_{i+1}) \le d$.
\end{proof}


\begin{proposition}[Full Rank]
\label{lem-rank} For a nonzero polynomial vector $\mathbf{a}$ of degree $d$,
defined by \eqref{eq-a}, such that $\gcd(\mathbf{a})=1$, the corresponding
matrix $A$, defined by \eqref{eq-A}, has rank $2d+1$.
\end{proposition}

\begin{proof}
By Lemma \ref{pre-rank}, for all $i=0,\ldots,2d$, there exist vectors
$\mathbf{h}_i\in\mathbb{K}[s]^{n}$ with $\deg(\mathbf{h}_i) \le d$ such that
$\mathbf{a}\mathbf{h}_i= s^{i}$. Observe that $(s^{i})^{\sharp} = e_{i+1}$.
Since $(\mathbf{a}\mathbf{h}_i)^{\sharp} = A\mathbf{h}_i^{\sharp}$, it follows
that there exist vectors $\mathbf{h}_i^{\sharp} \in\mathbb{K}^{n(d+1)}$ such
that $A\mathbf{h}_i^{\sharp} = e_{j}$ for all $j=1,\ldots,2d+1$. This means the
range of $A$ is $\mathbb{K}^{2d+1}$ and hence $\operatorname{rank}(A)=2d+1$.
\end{proof}

\subsection{The minimal B\'ezout vector theorem}

\label{ssect-bez} In this section, we construct a B\'ezout vector of $\aa$ of minimal
degree by finding an appropriate solution to the linear equation
\begin{equation}
\label{eq-bezA}A\,v=e_{1}, \text{ where } e_{1}=[1,0,\,\dots,\, 0]^{T}%
\in\mathbb{K}^{2d+1}.
\end{equation}
The following lemma establishes a one-to-one correspondence between the set
$\operatorname{Bez}_{d}(\mathbf{a})$ of B\'ezout vectors of $\aa$ of degree at most $d$ and the set
of solutions to \eqref{eq-bezA}.

\begin{lemma}
\label{lem-bez} Let $\mathbf{a}\in\mathbb{K}[s]_{d}^{n}$ be a nonzero vector
such that $\gcd(\mathbf{a})=1$. Then $\mathbf{b}\in\mathbb{K}[s]^{n}_{d}$
belongs to $\operatorname{Bez}_{d}(\mathbf{a})$ if and only if $\mathbf{b}%
^{\sharp}$ is a solution of \eqref{eq-bezA}. Also $v\in\mathbb{K}^{n(d+1)}$
solves \eqref{eq-bezA} if and only if $v^{\flat}$ belongs to
$\operatorname{Bez}_{d}(\mathbf{a})$.

\end{lemma}

\begin{proof}
Follows immediately from \eqref{eq-aA} and the observation that $e_{1}%
^{\flat^{1}_{2d}}=1$.
\end{proof}

Thus, our goal is to construct a solution $v$  of  \eqref{eq-bezA}, such that
$v^{\flat}$ is a {B\'ezout} vector of $\aa$ of minimal degree. To accomplish this,
we recall that, when $\gcd(\mathbf{a})=1$, Proposition~\ref{lem-rank} asserts
that $\mathrm{rank}(A)=2d+1$. Therefore, $A$ has exactly $2d+1$ pivotal
indices, which we can list in the increasing order $p=\{p_{1},\dots,
p_{2d+1}\}$. The corresponding columns of matrix $A$ form a basis of
$\mathbb{K}^{2d+1}$ and, therefore, $e_{1}\in\mathbb{K}^{2d+1}$ can be
expressed as a unique linear combination of the pivotal columns:
\begin{equation}
\label{eq-e1comb}e_{1}=\sum_{j=1}^{2d+1}\alpha_{j}A_{*p_{j}}.
\end{equation}
Define vector $v\in\mathbb{K}^{2d+1}$ by setting its $p_{j}$-th element to be
$\alpha_{j}$ and all other elements to be 0. We prove that $\mathbf{b}%
=v^{\flat}$ is a {B\'ezout} vector of $\aa$ of minimal degree.

\begin{theorem}
[Minimal-Degree {B\'ezout} Vector]\label{thm-bez} Let $\mathbf{a}\in\mathbb{K}%
[s]^{n}_{d}$ be a polynomial vector with $\gcd(\mathbf{a})=1$, and let $A$ be the
corresponding matrix defined by \eqref{eq-A}. Let $p=\{p_{1},\dots,
p_{2d+1}\}$ be the pivotal indices of $A$, and let $\alpha_{1},\dots
,\alpha_{2d+1}\in\mathbb{K}$ be defined by the unique expression
\eqref{eq-e1comb} of the vector $e_{1}\in\mathbb{K}^{2d+1}$ as a linear
combination of the pivotal columns of $A$. Define vector $v\in\mathbb{K}%
^{2d+1}$ by setting its $p_{j}$-th element to be $\alpha_{j}$ for $j=1,\dots,
2d+1$ and all other elements to be 0, and let $\mathbf{b}=v^{\flat}$. Then

\begin{enumerate}
\item $\mathbf{b}\in\operatorname{Bez}_d(\mathbf{a})$

\item $\deg(\mathbf{b})=\displaystyle{\min_{\mathbf{b}^{\prime}\in
\operatorname{Bez}(\mathbf{a})}} \deg(\mathbf{b}^{\prime}).$ 
\end{enumerate}
\end{theorem}


\begin{proof}\hfill
\begin{enumerate}
\item From \eqref{eq-e1comb}, it follows immediately that $Av=e_{1}$. 
Therefore, by Lemma~\ref{lem-bez}, we have that $\mathbf{b}=v^{\flat}%
\in\operatorname{Bez}_{d}(\mathbf{a})$.

\item To show that $\mathbf{b}$ is of minimal degree, we rewrite
\eqref{eq-e1comb} as
\begin{equation}
\label{eq-e1comb1}e_{1}=\sum_{j=1}^{k}\alpha_{j} A_{*p_{j}},
\end{equation}
where $k$ is the largest integer between 1 and $2d+1$, such that $\alpha
_{k}\neq0$. Then the last nonzero entry of $v$ appears in $p_{k}$-th position
and, therefore,
\begin{equation}
\label{degw}\deg(\mathbf{b})=\deg(v^{\flat})=\big\lceil{p_{k}/n}\big\rceil-1.
\end{equation}
Assume that $\mathbf{b}^{\prime}\in\operatorname{Bez}(\mathbf{a})$ is such
that $\deg(\mathbf{b}^{\prime})<\deg(\mathbf{b})$. Then $\mathbf{b}^{\prime
}\in\operatorname{Bez}_{d}(\mathbf{a})$ and therefore $A\,v^{\prime}=e_{1}$,
for $v^{\prime}=\mathbf{b}^{\prime\sharp}=[v^{\prime}_{1},\dots, v^{\prime
}_{n(d+1)}]\in\mathbb{K}^{n(d+1)}$. Then
\begin{equation}
\label{eq-e1v1}e_{1}=\sum_{j=1}^{n(d+1)}\,v^{\prime}_{j}A_{*j} =\sum_{j=1}%
^{r}v^{\prime}_{j} A_{*j},
\end{equation}
where $r$ is the largest integer between 1 and $n(d+1)$, such that $v^{\prime
}_{r}\neq0$.  Then
\begin{equation}
\label{degw'}\deg(\mathbf{b}^{\prime})=\big\lceil{r/n}\big\rceil-1
\end{equation}
and since we assumed that $\deg(\mathbf{b}^{\prime})<\deg(\mathbf{b})$, we
conclude from \eqref{degw} and \eqref{degw'} that $r<p_{k}$.

On the other hand, since all non-pivotal columns are linear combinations of
the preceding pivotal columns, we can rewrite \eqref{eq-e1v1} as
\begin{equation}
\label{eq-e1v2}e_{1}=\sum_{j\in\{1,\dots, 2d\,|\, p_{j}\leq r<p_{k}\}}%
\alpha^{\prime}_{j} A_{*p_{j}}=\sum_{j=1}^{k-1}\alpha^{\prime}_{j} A_{*p_{j}}.
\end{equation}
By the uniqueness of the representation of $e_1$ as a linear combination of the $A_{\ast p_j}$, the coefficients in the expansions \eqref{eq-e1comb1} and
\eqref{eq-e1v2} must be equal, but $\alpha_{k}\neq0$ in \eqref{eq-e1comb1}. Contradiction.
\end{enumerate}
\end{proof}

In the algorithm presented in Section~\ref{sect-alg}, we exploit the fact that
the coefficients $\alpha$'s in \eqref{eq-e1comb1} needed to construct a minimal-degree
{B\'ezout} vector of $\aa$ can be read off the reduced row echelon form $[\hat A| \hat
v]$ of the augmented matrix $[A|e_{1}]$. On the other hand, as was shown in \cite{hhk2017}
and reviewed in the next section, the coefficients of a $\mu$-basis of $\mathbf{a}$ also can be
read off the matrix $\hat A$. Therefore, a $\mu$-basis is constructed as a
byproduct of our algorithm for constructing a {B\'ezout} vector of minimal degree.


\subsection{ The $\mu$-bases theorem}

\label{ssect-mu}
In \cite{hhk2017}, we showed that the coefficients of a $\mu$-basis of $\aa$ can be read off the basic
non-pivotal columns of matrix $A$ (recall Definition~\ref{def-bnp}). Recall
that according to Lemma~\ref{lem-card}, the matrix $A$ has exactly $n-1$ basic
non-pivotal columns.

\begin{theorem}
[$\mu$-Basis]\label{thm-mu} Let $\mathbf{a}\in\mathbb{K}[s]^{n}_{d}$ be a
polynomial vector, and let $A$ be the corresponding matrix defined by \eqref{eq-A}.
Let $\tilde q=[\tilde q_{1},\dots, \tilde q_{n-1}]$ be the basic non-pivotal
indices of $A$, ordered increasingly. For $i=1,\dots,{n-1}$, a basic non-pivotal
column $A_{*\tilde q_{i}}$ is a linear combination of the previous pivotal
columns:
\begin{equation}
\label{eqAij}A_{*\tilde q_{i}}=\sum_{\{r\in p\,|\, r<\tilde q_{i}\}}
\alpha_{ir} A_{*r},
\end{equation}
for some $\alpha_{ir}\in\mathbb{K}$. Define vector $b_{i}\in\mathbb{K}^{2d+1}$
by setting its $\tilde q_{i}$-th element to be $1$, its $r$-th element to be
$-\alpha_{ir}$ for $r\in p$ such that $p_{j}<\tilde q_{i}$, and all other
elements to be 0. Then the set of polynomial vectors
\[
\mathbf{u}_{1}=b_{1}^{\flat},\quad\dots\quad, \mathbf{u}_{n-1}=b_{n-1}^{\flat}%
\]
is a degree-ordered $\mu$-basis of $\mathbf{a}$.
\end{theorem}


\begin{proof}
The fact that $\mathbf{u}_{1}=b_{1}^{\flat},\quad\dots\quad, \mathbf{u}%
_{n-1}=b_{n-1}^{\flat}$ is a $\mu$-basis of $\aa$ is the statement of Theorem~27 of
\cite{hhk2017}. By construction, the last nonzero entry of vector $b_{i} $ is
in the $\tilde q_{i}$-th position, and therefore for $i=1,\dots, n-1$,
\[
\deg(\mathbf{u}_{i}) =\deg(b_{i}^{\flat})=\big\lceil{\tilde q_{i}%
/n}\big\rceil-1.
\]
Since the indices in $\tilde q$ are ordered increasingly, the vectors
$\mathbf{u}_{1},\, \dots,\, \mathbf{u}_{n-1}$ are degree-ordered.
\end{proof}

The algorithm presented in Section~\ref{sect-alg} exploits the fact that the
coefficients $\alpha$'s in \eqref{eqAij} are already computed in the process
of computing a {B\'ezout} vector of $\aa$.

\section{The degree of an optimal moving frame}\label{sect-degree}
Similarly to the degree of a polynomial vector (Definition~\ref{def-basic}), we define the degree of a polynomial matrix to be the maximum of the degrees of its entries.  Obviously, for a given vector $\aa$, all degree-optimal moving frames have the same degree.  In this section, we establish the sharp upper and lower bounds on the degree of optimal moving frames. We also show that, for generic inputs, the degree of an optimal moving frame equals to the lower bound. 
{An alternative simple proof of the bounds  could be given using the fact that, when $\gcd(\aa)=1$, the sum of the degrees of  a $\mu$-basis of $\aa$ equals to $\deg(\aa)$ (see Theorem~2 in \cite{song-goldman-2009}), along with the result relating the degree of a minimal-degree \bez vector and the maximal degree of  a $\mu$-basis in Theorem~\ref{bez-mu} of the current paper. For the sharpness of the lower bound and its generality, one could use Proposition~3.3 of \cite{CI-2015},  determining the dimension of the set of vectors of a given $\mu$-type,   again combined with  Theorem~\ref{bez-mu} of the current paper. Our results on the upper bound differ from what can be deduced from \cite{CI-2015}, because we allow components of $\aa$ to be linearly dependent over $\k$.
To keep the presentation self-contained, we give  the proofs based on the results of the current paper.
We will repeatedly use the following lemma.
\begin{lemma}\label{lem-degree} Let $\aa\in\k[s]^n$ be nonzero and let $A$ be the corresponding matrix \eqref{eq-A}. Furthermore, let $k$ be the maximum among the basic non-pivotal indices of $A$. Then  the degree of any optimal moving frame at $\aa$ equals to $\left\lceil\frac{k}{n}\right\rceil-1$.
\end{lemma} 
\begin{proof} It is straightforward to check that the maximal degree of the $\mu$-basis, constructed in Theorem~\ref{thm-mu}, has degree $\left\lceil\frac{k}{n}\right\rceil-1$. From the optimality of the degrees property in Proposition~\ref{prop-equiv-def}, it follows that for any two degree-ordered $\mu$-bases $\uu_1,\dots, \uu_{n-1}$ and  $\uu'_1,\dots, \uu'_{n-1}$ of $\aa$ and for $i=1,\dots, n-1$, we have $\deg(\uu_i)=\deg(\uu_i')$. Therefore, the maximum  degree of vectors in any   $\mu$-basis equals to $\left\lceil\frac{k}{n}\right\rceil-1$. Theorem~\ref{bez-mu} implies that the degree of any optimal moving frame equals to the maximal degree of a $\mu$-basis.
\end{proof}
\begin{proposition}[Sharp Degree Bounds.]\label{deg} Let $\aa \in \mathbb{K}[s]^n$ with $\deg(\aa)=d$ and $\gcd(\aa)=1$. Then for every degree-optimal moving frame $P$ at $\aa$, we have $\lceil\frac{ d}{n-1}\rceil \le \deg(P) \le d$, and these degree bounds are sharp. By sharp, we mean that for all $n>1$ and $d>0$, there exists an $\aa \in \mathbb{K}[s]^n$ with $\deg(\aa)=d$ and $\gcd(\aa)=1$ such that, for every degree-optimal moving frame $P$ at $\aa$, we have $\deg(P) = \left\lceil\frac{d}{n-1}\right\rceil$. Likewise, for all $n>1$ and $d>0$, there exists an $\aa \in \mathbb{K}[s]^n$ with $\deg(\aa)=d$ and $\gcd(\aa)=1$ such that, for every degree-optimal moving frame $P$ at $\aa$, we have $\deg(P) = d$.
\end{proposition}

\begin{proof} \hfill

\begin{enumerate}
\item (lower bound): Let $P$ be a degree-optimal moving frame at $\aa$. Then 
$\aa P= 
\left[
\begin{array}{cccc}
  1  &
 0  &
 \cdots &
 0   
\end{array}
\right].
$
and so from Cramer's rule:
$$\aa_i= \frac{(-1)^{i+1}}{|P|}\,\left|P_{i,1}\right| \quad i=1,\dots n,$$
where $P_{i,1}$ denotes  the submatrix  of $P$ obtained by removing the $1$-st column and the $i$-th row. We remind the reader that $|P|$ is a nonzero constant. 
Assume for the sake of contradiction that $\deg(P) < \left\lceil\frac{ d}{n-1}\right\rceil$. Then $\deg(P)<\frac{ d}{n-1}$.  Since $\left|P_{i,1}\right|$ is the determinant of an $(n-1)\times (n-1)$ submatrix of $P$,  we have $\deg\left(\aa_i\right)=\deg\left(\left|P_{i,1}\right|\right) < (n-1)\,\frac{ d}{n-1}=d$ for all $i=1.\dots,n$.  This contradicts the assumption that $\deg(\aa)=d$. Thus, $\deg(P) \ge \left\lceil\frac{d}{n-1}\right\rceil$.

We will prove that the lower bound $\left\lceil\frac{d}{n-1}\right\rceil$ is sharp 
by showing that,  for all $n>1$ and $d>0$, the following matrix 
\beq
\mf = \left[\begin{array}
[c]{ccc|ccccc}%
1 &          &     & -s^{d-k\left\lceil\frac{d}{n-1}\right\rceil} &         &                 &      &   \\
  & \ddots &    &           &           &                &        & \\
  &            & 1 &           &           &                &         &\\\hline
  &             &    & 1       & -s^{\left\lceil\frac{d}{n-1}\right\rceil} &               &          &\\
   &             &    &       &   1       &  -s^{\left\lceil\frac{d}{n-1}\right\rceil}      &         & \\
  &             &    &        &             & \ddots    &    \ddots & \\
   &             &    &        &             &             &\ddots   &-s^{\left\lceil\frac{d}{n-1}\right\rceil} \\
     &             &    &        &             &             &         &1
\end{array}
\right]
\eeq
has degree $\left\lceil\frac{d}{n-1}\right\rceil$ and is a degree-optimal moving frame at the vector
\beq \aa = \left[1,0,\ldots,0,s^{d-k\cdot\left\lceil\frac{d}{n-1}\right\rceil},\ldots,s^{d-2\cdot\left\lceil\frac{d}{n-1}\right\rceil},s^{d-1\cdot\left\lceil\frac{d}{n-1}\right\rceil},s^{d-0\cdot\left\lceil\frac{d}{n-1}\right\rceil}\right].\eeq
Here  $k \in \mathbb{N}$ is the maximal such that $d>k\left\lceil\frac{d}{n-1}\right\rceil$ (explicitly $k=\left\lceil\frac d{\left\lceil\frac{d}{n-1}\right\rceil}\right\rceil -1$), the number of zeros in $\aa$   is $n-k-2$, the upper-left block of $P$ is of the size $(n-k-1)\times (n-k-1)$, the lower-right block is of the size $(k+1)\times (k+1)$, and the other two blocks are of the appropriate sizes.

First, we show that such $\aa$ and $P$ actually exist (not just optically). That is, the number of zeros in $\aa$ is non-negative, and the upper-left block in $P$ exists; 
 in other words, $n-1 \geq k+1$. Suppose otherwise. Then we would have
$$d-k\left\lceil\frac{d}{n-1}\right\rceil \le d-(n-1)\left\lceil\frac{d}{n-1}\right\rceil \le 0$$
which contradicts the condition  $d>k\left\lceil\frac{d}{n-1}\right\rceil$. 

Second, $P$ is a degree-optimal moving frame at $\aa$. Namely,
\begin{enumerate}
\item $\aa P = [1,0,\ldots,0]$, so $P$ is a moving frame at $\aa$.
\item The first column of $P$, $[1,0,\ldots,0]^T$, is a minimal-degree B\'ezout vector of $\aa$.
\item The last $n-1$ columns of $P$ are syzygies of $\aa$, and since $P \in \operatorname{mf}(\aa)$, by Proposition \ref{prop-mfbasis}, they form a basis of $\operatorname{syz}(\aa)$.  It is easy to see that these columns have linearly independent leading vectors as well. Thus, they form a $\mu$-basis of $\aa$.
\end{enumerate}

Finally, we show that the degree of $P$ is the lower bound, i.e. $\left\lceil\frac{d}{n-1}\right\rceil$. 
 From inspection of the entries  of $P$, we see immediately that $$\deg(P)=\max\left\{d-k\left\lceil\frac{d}{n-1}\right\rceil, \left\lceil\frac{d}{n-1}\right\rceil\right\}.$$ It remains to show that $d-k\left\lceil\frac{d}{n-1}\right\rceil \le \left\lceil\frac{d}{n-1}\right\rceil$. Suppose not.  Then $$d>(k+1)\left\lceil\frac{d}{n-1}\right\rceil,$$ a contradiction to the maximality of $k$. %
Thus, $\deg(P)=\left\lceil\frac{ d}{n-1}\right\rceil$. Hence, we have proved that the lower bound is sharp.

\item (upper bound): From Theorems~\ref{thm-bez} and~\ref{thm-mu}, it follows immediately that $d$ is an upper bound of a degree-optimal moving frames. We will prove that the upper bound $d$  is sharp
by showing that,  for all $n>1$ and $d>0$, the following matrix of degree~$d$
\beq
P =\left[
\begin{array}[c]{ccccc}
1  &  &  &  & -s^{d}\\
   &  \ddots &  &  & \\
   &  & \ddots &  & \\
   &  & & \ddots & \\
   &  & &  & 1
\end{array}
\right].
\eeq
is the degree-optimal moving frame for the vector
$$ \aa = [1, 0,\dots, 0,s^d]$$
Indeed:
\begin{enumerate}
\item $\aa P = [1,0,\ldots,0]$ and  so $P$ is a moving frame at $\aa$.
\item The first column of $P$, $[1,0,\ldots,0]^T$, is a minimal-degree B\'ezout vector of $\aa$.
\item The last $n-1$ columns of $P$ are syzygies of $\aa$, and since $P \in \operatorname{mf}(\aa)$, by Proposition \ref{prop-mfbasis}, they form a basis of $\operatorname{syz}(\aa)$.  It is easy to see that these columns have linearly independent leading vectors as well. Thus, they form a $\mu$-basis of $\aa$.
\end{enumerate}
\end{enumerate}
\end{proof}

In Theorem \ref{generic-deg} below, we show that for generic $\aa \in \mathbb{K}[s]^n$ with $\deg(\aa)=d$ and $\gcd(\aa)=1$, and for all degree-optimal moving frames $P$ at $\aa$, $\deg(P) = \left\lceil \frac{d}{n-1} \right\rceil$. To prove the theorem, we need the following lemmas, where  we will use notation $$k = \operatorname{quo}(d,n-1)\text{ and }r = \operatorname{rem}(d,n-1).$$
\begin{lemma}\label{Cmatrix}
For arbitrary $\aa \in \mathbb{K}[s]^n$ with $\deg(\aa)=d$ and $\gcd(\aa)=1$, the principal $d+k+1$ submatrix of the associated matrix $A$ has the form
\begin{equation}\label{generic-matrix}
C = \left[
\begin{array}
[c]{cccccccccccc}%
c_{01} & \cdots & \cdots & c_{0n} &  &  &  &  &  &  & \\
\vdots & \cdots & \cdots & \vdots & c_{01} & \cdots & \cdots & c_{0n} &  &  &  & \\
\vdots & \cdots & \cdots & \vdots & \vdots & \cdots & \cdots & \vdots & \ddots &  &  & \\
c_{d1} & \cdots & \cdots & c_{dn} & \vdots & \cdots & \cdots & \vdots & \ddots & c_{01} & \cdots & c_{0,r+1}\\
&  & &  & c_{d1} & \cdots & \cdots & c_{dn} & \ddots & \vdots & \cdots & \vdots\\
& &  & &  &  &  &  & \ddots & \vdots & \cdots & \vdots\\
& & & & &  &  &  &  & c_{d1} & \cdots & c_{d,r+1}%
\end{array}
\right],
\end{equation}
where $C$ consists of $k$ full $(d+1)\times n$ size blocks and 1 partial block of size $(d+1)\times(r+1)$.
\end{lemma}
\begin{proof}
If we take $k$ full $(d+1)\times n$ blocks and 1 partial $(d+1)\times(r+1)$ block, then the number of columns of $C$ is $nk+r+1 = (n-1)k+r+k+1 = d+k+1$, as desired. Furthemore, since the leftmost block takes up the first $d+1$ rows of $C$, and we shift the block down by 1 a total of $k$ times, the number of rows of $C$ is $d+k+1$ as well.
\end{proof}
\begin{lemma}\label{Cnonsingular}
Let $\aa \in \mathbb{K}[s]^n$ with $\deg(\aa)=d$ and $\gcd(\aa)=1$, and let $C$ be the principal $d+k+1$ submatrix of $A$ given by \eqref{generic-matrix}. If $C$ is nonsingular, then for any degree-optimal moving frame $\mf$ at $\aa$, we have $\deg(P) = \left\lceil\frac{d}{n-1}\right\rceil$.
\end{lemma}
\begin{proof}
If $C$ is nonsingular, then first $d+k+1$ columns of the matrix $A$ are pivotal columns. Since $\rank (A) =2d+1$, there are $d-k$ additional pivotal columns in $A$ and, from the structure of $A$,  each of the last $d-k$  blocks of $A$ contain exactly one of these additional pivotal columns. All other columns in $A$ are non-pivotal.   
We now consider two cases:
\begin{enumerate}
\item[1)] If $n-1$ divides $d$, then
 $r=0$ and $k = \frac{d}{n-1} = \left\lceil\frac{d}{n-1}\right\rceil$. Thus, there is one column in the partial block in $C$, and so the remaining $n-1$ columns in this $(k+1)$-th block of $A$ are basic non-pivotal columns. Since in total there are $n-1$ basis non-pivotal columns, the largest basic non-pivotal index equals to  $n(k+1)$, and therefore by Lemma~\ref{lem-degree},  the degree of any optimal moving frame at $\aa$ is $\left\lceil\frac{d}{n-1}\right\rceil$.
\item[2)]  If $n-1$ does not divide $d$, then
 $r > 0$ and $k = \left\lfloor\frac{d}{n-1}\right\rfloor$. Thus, there are at least two columns in the partial block in $C$, and so there are at most $n-2$ basic non-pivotal columns in the $(k+1)$-th block of $A$. Since there are a total of $n-1$ basis non-pivotal columns, and all but one of the columns in the $(k+2)$-th block are non-pivotal, the largest basic non-pivotal column  index will appear in the $(k+2)$-th block. Therefore, this  largest index equals    to $n(k+1)+j$  for some $1\le j \le n$.  By Lemma~\ref{lem-degree},  the degree of any optimal moving frame at $\aa$ equals to $\left\lceil\frac{n(k+1)+j}{n}\right\rceil-1 = k+1 = \left\lfloor\frac{d}{n-1}\right\rfloor+1 = \left\lceil\frac{d}{n-1}\right\rceil$. 
 \end{enumerate}   
\end{proof}
\begin{lemma}\label{det-nonzero}
For all $n>1$ and $d>0$, there exists a vector $\aa \in \mathbb{K}[s]^n$ with $\deg(\aa)=d$ and $\gcd(\aa)=1$ such that $\det(C) \not= 0$.

\end{lemma}
\begin{proof}
Let $n>1$ and $d>0.$ We will find a suitable witness for $\mathbf{a}$.
Recalling the relation $d=k\left( n-1\right) +r$, we will consider the
following three cases:

\begin{enumerate}
\item[1)] If $n-1>d$, we claim that the following is a witness: 
\begin{equation*}
\mathbf{a}=\left[s^{d},s^{d-1},\ldots ,s,1,\ldots ,1\right]
\end{equation*}%
Note that there is at least one  $1$ at the end. Thus  $\deg (\mathbf{a})=d$ and $\gcd (\mathbf{a})=1$. It
remains to show that $\left\vert C\right\vert \neq 0$. Note that $k=0$ and $%
r=d.$ Thus, the matrix $C$ is a $(d+1)\times (d+1)$ partial block that looks
like 
\begin{equation*}
C=\left[ 
\begin{array}{ccc}
&  & 1 \\ 
& \iddots &  \\ 
1 &  & 
\end{array}%
\right] .
\end{equation*}%
Therefore, $|C|=\pm 1$.

\item[2)] If $n-1\leq d$ and $n-1$ divides $d$, we claim that the following is
a witness: 
\begin{equation*}
\mathbf{a}=\left[ s^{d},s^{d-k},\ldots ,s^{d-(n-1)k}\right] 
\end{equation*}%
Note that the last component is $s^{d-(n-1)k}=s^{0}=1$. Thus $\deg (\mathbf{a})=d$ and $\gcd (\mathbf{a})=1$. It remains to show
that $\left\vert C\right\vert \neq 0$. To do this, we examine the shape of $C
$. To get intuition, consider the instance where $n=3$ and $d=6$. Note that $k=3$
and $r=0$. Thus, we have%
\begin{eqnarray*}
a &=&\left[ s^{6},s^{3},s^{0}\right]  \\
C &=&\left[ 
\begin{array}{cccccccccc}
0 & 0 & \mathbf{1} &  &  &  &  &  &  &  \\ 
0 & 0 & 0 & 0 & 0 & \mathbf{1} &  &  &  &  \\ 
0 & 0 & 0 & 0 & 0 & 0 & 0 & 0 & \mathbf{1} &  \\ 
0 & \mathbf{1} & 0 & 0 & 0 & 0 & 0 & 0 & 0 & 0 \\ 
0 & 0 & 0 & 0 & \mathbf{1} & 0 & 0 & 0 & 0 & 0 \\ 
0 & 0 & 0 & 0 & 0 & 0 & 0 & \mathbf{1} & 0 & 0 \\ 
\mathbf{1} & 0 & 0 & 0 & 0 & 0 & 0 & 0 & 0 & 0 \\ 
&  &  & \mathbf{1} & 0 & 0 & 0 & 0 & 0 & 0 \\ 
&  &  &  &  &  & \mathbf{1} & 0 & 0 & 0 \\ 
&  &  &  &  &  &  &  &  & \mathbf{1}%
\end{array}%
\right] 
\end{eqnarray*}%
All the empty spaces are zeros. Note that $C$ is a permutation matrix (each
row has only one $1$ and each column has only one $1$). Therefore, $|C|=\pm 1
$. It is easy to see that the same observation holds in general.

\item[3)] If $n-1\leq d$ and $n-1$ does not divide $d$, we claim that the
following is a witness: 
\begin{equation*}
\mathbf{a}=\left[ s^{d},s^{d-(1k+1)},s^{d-(2k+2)}\ldots
,s^{d-(rk+r)},s^{d-((r+1)k+r)},\ldots ,s^{d-((n-1)k+r)}\right] 
\end{equation*}

Note that the last component is $s^{d-((n-1)k+r)}=s^{0}=1$. Thus $\deg (\mathbf{a})=d$ and $\gcd (\mathbf{a})=1$. It remains to
show that $\left\vert C\right\vert \neq 0$. To do this, we examine the shape
of $C$. To get intuition, consider the case $n=5$ and $d=14$. Note that $k=3$ and $%
r=2.$ Thus, we have%
\begin{small}\begin{eqnarray*}
a &=&\left[ s^{14},s^{14-\left( 1\cdot 3+1\right) },s^{14-\left( 2\cdot
3+2\right) },s^{14-\left( 3\cdot 3+2\right) },s^{14-\left( 4\cdot 3+2\right)
}\right] =\left[ s^{14},s^{10},s^{6},s^{3},s^{0}\right]  \\
C &=&\left[ 
\begin{array}{cccccccccccccccccc}
0 & 0 & 0 & 0 & \mathbf{1} &  &  &  &  &  &  &  &  &  &  &  &  &  \\ 
0 & 0 & 0 & 0 & 0 & 0 & 0 & 0 & 0 & \mathbf{1} &  &  &  &  &  &  &  &  \\ 
0 & 0 & 0 & 0 & 0 & 0 & 0 & 0 & 0 & 0 & 0 & 0 & 0 & 0 & \mathbf{1} &  &  & 
\\ 
0 & 0 & 0 & \mathbf{1} & 0 & 0 & 0 & 0 & 0 & 0 & 0 & 0 & 0 & 0 & 0 & 0 & 0 & 
0 \\ 
0 & 0 & 0 & 0 & 0 & 0 & 0 & 0 & \mathbf{1} & 0 & 0 & 0 & 0 & 0 & 0 & 0 & 0 & 
0 \\ 
0 & 0 & 0 & 0 & 0 & 0 & 0 & 0 & 0 & 0 & 0 & 0 & 0 & \mathbf{1} & 0 & 0 & 0 & 
0 \\ 
0 & 0 & \mathbf{1} & 0 & 0 & 0 & 0 & 0 & 0 & 0 & 0 & 0 & 0 & 0 & 0 & 0 & 0 & 
0 \\ 
0 & 0 & 0 & 0 & 0 & 0 & 0 & \mathbf{1} & 0 & 0 & 0 & 0 & 0 & 0 & 0 & 0 & 0 & 
0 \\ 
0 & 0 & 0 & 0 & 0 & 0 & 0 & 0 & 0 & 0 & 0 & 0 & \mathbf{1} & 0 & 0 & 0 & 0 & 
0 \\ 
0 & 0 & 0 & 0 & 0 & 0 & 0 & 0 & 0 & 0 & 0 & 0 & 0 & 0 & 0 & 0 & 0 & \mathbf{1%
} \\ 
0 & \mathbf{1} & 0 & 0 & 0 & 0 & 0 & 0 & 0 & 0 & 0 & 0 & 0 & 0 & 0 & 0 & 0 & 
0 \\ 
0 & 0 & 0 & 0 & 0 & 0 & \mathbf{1} & 0 & 0 & 0 & 0 & 0 & 0 & 0 & 0 & 0 & 0 & 
0 \\ 
0 & 0 & 0 & 0 & 0 & 0 & 0 & 0 & 0 & 0 & 0 & \mathbf{1} & 0 & 0 & 0 & 0 & 0 & 
0 \\ 
0 & 0 & 0 & 0 & 0 & 0 & 0 & 0 & 0 & 0 & 0 & 0 & 0 & 0 & 0 & 0 & \mathbf{1} & 
0 \\ 
\mathbf{1} & 0 & 0 & 0 & 0 & 0 & 0 & 0 & 0 & 0 & 0 & 0 & 0 & 0 & 0 & 0 & 0 & 
0 \\ 
&  &  &  &  & \mathbf{1} & 0 & 0 & 0 & 0 & 0 & 0 & 0 & 0 & 0 & 0 & 0 & 0 \\ 
&  &  &  &  &  &  &  &  &  & \mathbf{1} & 0 & 0 & 0 & 0 & 0 & 0 & 0 \\ 
&  &  &  &  &  &  &  &  &  &  &  &  &  &  & \mathbf{1} & 0 & 0%
\end{array}%
\right] 
\end{eqnarray*}%
\end{small}
All the empty spaces are zeros. Note that $C$ is a permutation matrix (each
row has only one $1$ and each column has only one $1$). Therefore, $|C|=\pm 1
$. It is easy to see that the same observation holds in general.
\end{enumerate}
\end{proof}

\begin{theorem}[Generic Degree.]\label{generic-deg}
Let $\mathbb{K}$ be an infinite field. For generic $\aa \in \mathbb{K}[s]^n$ with $\deg(\aa)=d$ and $\gcd(\aa)=1$, for every degree-optimal moving frame $P$ at $\aa$, we have $\deg(P) = \left\lceil\frac{d}{n-1}\right\rceil$.
\end{theorem}
\begin{proof}
From Lemma~\ref{det-nonzero}, it follows that $\det(C)$ is a nonzero polynomial on the $n(d+1)$-dimensional vector space $\k[s]^n$ over $\k$. Thus, the condition $\det(C)\neq 0$ defines a proper Zariski open subset of $\k[s]^n$.  Lemma \ref{Cnonsingular} implies that for every $\aa$ in this Zariski open subset, every degree-optimal moving frame $P$ at $\aa$ has degree $\left\lceil\frac{d}{n-1}\right\rceil$. If we assume $\mathbb{K}$ is an infinite field, then the complement of any proper Zariski open subset is of measure zero, and we can say that for a generic $\aa$, the degree of every degree-optimal moving frame at $\aa$ equals the sharp lower bound $\left\lceil\frac{d}{n-1}\right\rceil$.
\end{proof}

{\Blue
\begin{remark} Some simple consequences of the general results about the degrees  are  worthwhile  recording.  From Proposition~\ref{deg}, it follows that, when $d\geq n$, the  degree of an optimal  moving frame is always strictly greater than 1.  From the above theorem and Theorem~\ref{bez-mu}, it follows that when $d<n$ and $\k$  is infinite, then for a 
generic input, the   degree of an optimal  moving frame is 1 and   the minimal-degree \bez vector is a constant vector.  
\end{remark}
}

\section{The OMF-Algorithm}

\label{sect-alg}
The theory developed in Sections \ref{sect-mf-bv-syz} and \ref{sect-la} can be recast into an algorithm
for computing a degree-optimal moving frame.
In this section, $\mathrm{quo}(i,j)$ denotes the quotient and $\mathrm{rem}%
(i,j)$ denotes the remainder generated by dividing an integer $i$ by an
integer $j$.

\begin{algorithm}[\textbf{OMF}]\hfill
\label{alg-OMF}
\begin{description}
\item[\textit{Input:}] $\mathbf{a }\not = 0 \in\mathbb{K}[s]^{n}$, row vector,
where $n>1$, $\gcd(\mathbf{a})=1$, and $\mathbb{K}$ a computable field

\item[\textit{Output:}] $P\in\mathbb{K}[s]^{n\times n}$, a degree-optimal moving
frame at $\mathbf{a}$
\end{description}

\begin{enumerate}
\item \emph{Construct a matrix $W \in\mathbb{K}^{(2d+1)\times(nd+n+1)} $,
whose left $(2d+1)\times (nd+n)$ block is matrix \eqref{eq-A} and whose last
column is $e_{1}$.} 

\begin{enumerate}

\item $d\longleftarrow\deg(\mathbf{a})$

\item Identify the row vectors $c_{0}=[c_{01}, \dots c_{0n}],\ldots,c_{d}=
[c_{d1}, \dots c_{dn}]$ such that $\mathbf{a}=c_{0}+c_{1}s+\cdots+c_{d}s^{d}$.

\item $W\longleftarrow\left[
\begin{array}
[c]{ccc}%
c_{0} &  & \\
\vdots & \ddots & \\
c_{d} & \vdots & c_{0}\\
& \ddots & \vdots\\
&  & c_{d}%
\end{array}
\right|  \left.
\begin{array}
[c]{c}%
1\\
0\\
\\
\vdots\\
\\
0
\end{array}
\right] \in\mathbb{K}^{(2d+1)\times(nd+n+1)}$

\end{enumerate}

\item \emph{Construct the ``partial'' reduced row-echelon form $E$ of $W$.}

This can be done by using Gauss-Jordan elimination (forward elimination,
backward elimination, and normalization), with the following optimizations:

\begin{itemize}
\item Skip over periodic non-pivot columns.

\item Carry out the row operations only on the required columns.
\end{itemize}

\item \emph{Construct a matrix $P\in\mathbb{K}[s]^{n\times n}$ whose first
column is a {B\'ezout} vector of $\aa$ of minimal degree and whose last $n-1$ columns form
a $\mu$-basis of $\mathbf{a}$.}

Let $p$ be the list of the pivotal indices and let $\tilde{q}$ be the list of
the basic non-pivotal indices of $E$.

\begin{enumerate}
\item Initialize an $n\times n$ matrix $P$ with $0$ in every entry.

\item For $j=2,\ldots,n$

\qquad$r\leftarrow\operatorname*{rem}\left(  \tilde{q}_{j-1}-1,n\right)  +1$

\qquad$k\leftarrow\operatorname*{quo}\left(  \tilde{q}_{j-1}-1,n\right)  $

\qquad$P_{r,j\ }\leftarrow P_{r,j}+s^{k}$

\item For $i=1,\ldots,2d+1$

\qquad$r\leftarrow\operatorname*{rem}\left(  p_{i}-1,n\right)  +1$

\qquad$k\leftarrow\operatorname*{quo}\left(  p_{i}-1,n\right)  $

\qquad$P_{r,1} \leftarrow P_{r,1}+E_{i,nd+n+1}s^{k}$

\qquad For $j=2,\ldots,n$

\qquad$\qquad P_{r,j}\leftarrow P_{r,j}-E_{i,\tilde{q}_{j-1}}s^{k}$
\end{enumerate}
\end{enumerate}
\end{algorithm}

\begin{remark}\rm
Step 3 of the OMF algorithm consists of constructing the moving frame $P$ from the entries of $E$. This step can be completed by explicitly constructing the nullspace vectors of $A$ corresponding to the $n-1$ basic non-pivotal columns of $E$ and the solution vector $v$ to $Av = e_1$ corresponding to the last column of $E$; and then translating these vectors into polynomial vectors using the $\flat$ isomorphism. However, this does some wasteful operations. The matrix $E$ contains all of the information needed to construct $P$, so we only need to read off the desired entries of $E$ instead of constructing entire vectors. This is what is done in step 3. Step 3(b) computes the leading polynomial entry for each $\mu$-basis column corresponding to the index of the corresponding basic non-pivotal column, while step 3(c) computes the remaining entries in the $\mu$-basis columns and the entries of the B\'ezout vector column corresponding to the indices of the pivot columns.
\end{remark}

\begin{theorem}
The output of the OMF Algorithm is a degree-optimal moving frame at $\mathbf{a}$,
where $\mathbf{a}$ is the input vector $\mathbf{a }\in\mathbb{K}[s]^{n}$ such
that $n>1$ and $\gcd(\mathbf{a})=1$.
\end{theorem}

\begin{proof}
In step 1, we construct a matrix $W=[A\text{ $|$ }e_{1}] \in\mathbb{K}^{(2d+1)\times
(nd+n+1)} $ whose left $(2d+1)\times(nd+n)$ block is matrix \eqref{eq-A} and
whose last column is $e_{1}=[1,0,\dots,0]^{T}$. Under isomorphism $\flat$, the
null space of $A$ corresponds to $\operatorname{syz}_{d}(\mathbf{a})$, and the
solutions to $Av=[1,0,\ldots,0]^{T}$ correspond to $\operatorname{Bez}%
_{d}(\mathbf{a})$. From Proposition~\ref{lem-rank}, we know that ${\mathrm{rank}}(A)=2d+1$, and thus all pivotal columns of $W$ are the pivotal columns of
$A$. In step 2, we perform partial Gauss-Jordan operations on $W$ to identify
the coefficients $\alpha$'s appearing in \eqref{eqAij} and \eqref{eq-e1comb},
that express the $n-1$ basic non-pivotal columns of $A$ and the vector $e_{1}$,
respectively, as linear combinations of pivotal columns of $A$. These
coefficients will appear in the basic non-pivotal columns and the last column
of the partial reduced row-echelon form $E$ of $W$. In Step 3, we use these
coefficients to construct a minimal-degree {B\'ezout} vector of $\aa$ and a degree-ordered
$\mu$-basis of $\aa$, as prescribed by Theorems~\ref{thm-bez} and \ref{thm-mu}. We place
these vectors as the columns of matrix $P$, and the resulting matrix is, indeed, a degree-optimal moving frame according to Theorem~\ref{thm-bbomf}.
\end{proof}



\begin{example}
We trace the algorithm on the input vector
\[
\aa=\left[
\begin{array}
[c]{ccc}%
2+s+s^{4} & 3+s^{2}+s^{4} & 6+2s^{3}+s^{4}%
\end{array}
\right]  \in\mathbb{Q}[s]^{3}
\]
\begin{enumerate}
\newcolumntype{C}{>{\small\raggedleft\color{black}\arraybackslash$}p{1.2em}<{$}}
\newcolumntype{G}{>{\small\raggedleft\color{gray}\arraybackslash$}p{1.2em}<{$}}
\newcolumntype{B}{>{\small\raggedleft\color{blue}\arraybackslash$}p{1.2em}<{$}}
\newcolumntype{R}{>{\small\raggedleft\color{red}\arraybackslash$}p{1.6em}<{$}}
\newcolumntype{T}{>{\small\raggedleft\color{brown}\arraybackslash$}p{1.2em}<{$}}

\item \textrm{\emph{Construct matrix $W$}}:

\begin{enumerate}

\item \textrm{$d\longleftarrow4$ }

\item \textrm{$c_{0},c_{1},c_{2},c_{3},c_{4}\longleftarrow\left[
\begin{array}
[c]{ccc}%
2 & 3 & 6
\end{array}
\right]  ,\left[
\begin{array}
[c]{ccc}%
1 & 0 & 0
\end{array}
\right]  ,\left[
\begin{array}
[c]{ccc}%
0 & 1 & 0
\end{array}
\right]  ,\left[
\begin{array}
[c]{ccc}%
0 & 0 & 2
\end{array}
\right]  ,\left[
\begin{array}
[c]{ccc}%
1 & 1 & 1
\end{array}
\right]  $ }

\item $W \longleftarrow\left[
\begin{array}{CCC|CCC|CCC|CCC|CCC|C}
2&3&6&&&&&&&&&&&&&1\\
1&0&0&2&3&6&&&&&&&&&&\\
0&1&0&1&0&0&2&3&6&&&&&&&\\
0&0&2&0&1&0&1&0&0&2&3&6&&&&\\
1&1&1&0&0&2&0&1&0&1&0&0&2&3&6&\\
&&&1&1&1&0&0&2&0&1&0&1&0&0&\\
&&&&&&1&1&1&0&0&2&0&1&0&\\
&&&&&&&&&1&1&1&0&0&2&\\
&&&&&&&&&&&&1&1&1&\end{array}
\right] $
\end{enumerate}

\item \textrm{\emph{Construct the ``partial'' reduced row-echelon form $E$
of $W$.}\newline$E \longleftarrow\left[
\begin{array}{BBB|BBB|BRR|BTT|BTT|G}
1&&&&&&&-3&-9&&&&&&&2\\
&1&&&&&&-2&-8&&&&&&&1\\
&&1&&&&&2&7&&&&&&&-1\\
&&&1&&&&3&12&&3&6&&&&-1\\
&&&&1&&&-5&\tiny-15&&0&0&&3&6&2\\
&&&&&1&&2&5&&1&0&&0&0&-1\\
&&&&&&1&1&1&&0&2&&1&0&0\\
&&&&&&&&&1&1&1&&0&2&0\\
&&&&&&&&&&&&1&1&1&0\end {array}
\right] $ \newline \emph{Here, blue denotes pivotal columns, red denotes basic non-pivotal columns, brown denotes periodic non-pivotal columns, and grey denotes the solution column.}}

\item \textrm{\emph{Construct a matrix $P \in\mathbb{K}[s]^{n\times n}$
whose first column consists of a minimal-degree {B\'ezout} vector of $\mathbf{a}$ and whose
last $n-1$ columns form a $\mu$-basis of $\mathbf{a}$.} }

\begin{enumerate}
\item \textrm{$P \longleftarrow\left[
\begin{array}
[c]{ccc}%
0 & 0 & 0\\
0 & 0 & 0\\
0 & 0 & 0
\end{array}
\right]  $ }

\item \textrm{$P \longleftarrow\left[
\begin{array}
[c]{ccc}%
0 & 0 & 0\\
0 & {s}^{2} & 0\\
0 & 0 & {s}^{2}%
\end{array}
\right]  $ }

\item \textrm{$P\longleftarrow\left[
\begin{array}
[c]{ccc}%
2-s & 3-3\,s-{s}^{2} & 9-12\,s-{s}^{2}\\
1+2\,s & 2+5\,s+{s}^{2} & 8+15\,s\\
-1-s & -2-2\,s & -7-5\,s+{s}^{2}%
\end{array}
\right]  $}
\end{enumerate}
\end{enumerate}
\end{example}


\begin{proposition}
[Theoretical Complexity] \label{prop-complexity} Under the assumption that the time for any arithmetic operation is
constant, the complexity of the OMF algorithm is
\(
O(d^{2}n + d^{3} + n^{2}).
\)

\end{proposition}

\begin{proof}
We will trace the theoretical complexity for each step of the algorithm.

\begin{enumerate}
\item
\begin{enumerate}
\item To determine $d$, we scan through each of the $n$ polynomials in $a$ to
identify the highest degree term, which is always $\leq d$. Thus, the
complexity for this step is $O(dn)$.

\item We identify $n(d+1)$ values to make up $c_{0},\ldots,c_{d}$. Thus, the
complexity for this step is $O(dn)$.

\item We construct a matrix with $(2d+1)(nd+n+1)$ entries. Thus, the
complexity for this step is $O(d^{2}n)$.
\end{enumerate}

\item With the partial Gauss-Jordan elimination, we perform row operations
only on the $2d+1$ pivot columns of $A$, the $n-1$ basic non-pivot columns of
$A$, and the augmented column $e_1$. Thus, we perform Gauss-Jordan elimination
on a $(2d+1)\times(2d+n+1)$ matrix. In general, for a $k\times l$ matrix,
Gauss-Jordan elimination has complexity $O(k^{2}l)$. Thus, the complexity for
this step is $O(d^{2}(d+n))$.

\item
\begin{enumerate}
\item We fill 0 into the entries of an $n\times n$ matrix $P$. Thus, the
complexity for this step is $O(n^{2})$.

\item We update entries of the matrix $n-1$ times. Thus, the complexity for
this step is $O(n)$.

\item We update entries of the matrix $(2d+1)(n-1)$ times. Thus, the
complexity for this step is $O(dn)$.
\end{enumerate}
\end{enumerate}

\noindent By summing up, we have
\(
O\left(  dn+dn+d^{2}n+d^{2}(d+n)+n^{2}+n+dn\right)   =O\left(  d^{2}%
n+d^{3}+n^{2} \right)
\)
\end{proof}

\begin{remark}
\label{rem-sparse}\textrm{Note that the $n^{2}$ term in the above complexity
is solely due to step 3(a), where the matrix $P$ is initialized with zeros. If
one uses a \emph{sparse\/} representation of the matrix (storing only nonzero
elements), then one can skip the initialization of the matrix $P$. As a
result, the complexity can be improved to $O\left( d^{2}n+d^{3}\right) $. }
\end{remark}
{It turns out that the theoretical complexity of the OMF algorithm is
exactly the same as that of the $\mu$-basis algorithm presented in \cite{hhk2017}.
This is unsurprising, because the $\mu$-basis algorithm presented
in \cite{hhk2017} is based on partial Gauss-Jordan elimination of matrix $A$,
while the OMF algorithm is based on partial Gauss-Jordan elimination of the matrix obtained
by appending to $A$ a single column $e_{1}$.}

\section{Comparison with other approaches}
\label{sect-compare}
We are not aware of any previously published algorithm for degree-optimal moving frames. Hence, we cannot compare the algorithm OMF\ with any existing algorithms. Instead, we  will compare with a not yet published, but tempting alternative approach. The approach consists of two steps: 
(1) Compute a  moving frame.
(2) Reduce the degree to obtain a degree-optimal moving frame.
We elaborate on this \emph{two-step} approach.
\begin{enumerate}
\item[(1)]\emph{Compute a moving frame.}  A non-optimal moving frame can be computed by a variety of methods, and in particular in the process  of computing normal forms of polynomial matrices, such as     in \cite{beckermann-labahn-villard-1999}, \cite{beckermann-labahn-villard-2006}.
The problem of constructing  an algebraic moving frame is also a particular case of the well-known problem of providing  a constructive proof
of the Quillen-Suslin theorem \cite{fitchas-galligo-1990}, \cite{logar-sturmfels-1992}, \cite{caniglia-1993}, \cite{lombardi-yengui-2005}, \cite{fabianska-quadrat-2007}.  In those papers, the multivariate problem is reduced inductively to the univariate case, and then an algorithm for the univariate case is proposed. Those univariate algorithms produce moving frames. 
As far as we are aware, the produced moving frames are \emph{usually not} degree-optimal.  
However, the  algorithms are very efficient. We will work with one such algorithm
used by Fabianska and Quadrat in \cite{fabianska-quadrat-2007},
because it has the least computational complexity among algorithms of which we are aware. 
Furthermore, the algorithm has been implemented  by the authors in {\sc Maple}, 
and  the package can be obtained from  \url{http://wwwb.math.rwth-aachen.de/QuillenSuslin/}. 
For the readers' convenience, we outline their algorithm (for univariate case)\  below:
 \begin{enumerate}
\item[(a)] Find constants $k_3,\ldots,k_n$ such that $\gcd(a_1+k_3a_3+\cdots+k_na_n,a_2)=1$.
\item[(b)] Find $f_1, f_2 \in \mathbb{K}[s]$ such that 
$(a_1+k_3a_3+\cdots+k_na_n)f_1+a_2f_2=1$.
This can be done by using the Extended Euclidean Algorithm.
\item[(c)] 
$P \longleftarrow 
\left[\begin{array}{ccccc}
1 & & & &\\
& 1 & & &\\
k_3 & & 1 & &\\
\vdots & & & \ddots &\\
k_n & & & & 1
\end{array}\right]
\left[\begin{array}{ccccc} 
f_1 & -a_2 & & &\\
f_2 & a_1' & & &\\
& & 1 & &\\
& & & \ddots &\\
& & & & 1
\end{array}\right]
\left[\begin{array}{ccccc} 
1 & 0 & -a_3 & \cdots & -a_n\\
0 & 1 & & &\\
& & 1 & &\\
& & & \ddots &\\
& & & & 1
\end{array}\right]$,\\
where $a_1' = a_1+k_3a_3+\cdots+k_na_n$.
\end{enumerate}
We remark that Step (a) of this algorithm can be completed with a random search. Moreover, for random inputs, $\gcd(a_1,a_2)=1$ and one can take each $k_i =0$. The complexity of this algorithm is $O(d^2+n^3)$, where $d^2$ comes from the Extended Euclidean Algorithm and $n^3$ comes from forming the matrix $P$, which is much better than the complexity of the OMF algorithm. We note, however, that the output of the Fabianska-Quadrat algorithm has degree at least $d$, while the output of the OMF algorithm has degree at most $d$ and generically   $\left\lceil\frac{d}{n-1}\right\rceil$.    
 
\item[(2)]  \emph{Reduce the degree to obtain a degree-optimal moving frame.} 
There are several different ways to carry out degree reduction: 
Popov form (\cite{beckermann-labahn-villard-1999}, \cite{beckermann-labahn-villard-2006}),  column reduced form \cite{cheng-labahn-2007}
and  matrix GCD \cite{beckermann-labahn-2000}.
  As far as we are aware, the Popov form algorithm  \cite{beckermann-labahn-villard-2006} is the only one with a publicly available Maple implementation.  Thus, we will use it for comparison. We explain  how to use Popov form to reduce the degree.  

\begin{enumerate}
\item[(a)] Compute the Popov normal form of the last $n-1$ columns of a non-optimal moving frame
$P$.
\item[(b)] Reduce the degree of the first  column of $P$ (a \bez vector) by the  Popov normal form of the last $n-1$ columns.
\end{enumerate}  
\end{enumerate}

\noindent We  compared the computing times of 
the algorithm OMF and the alternative two-step approach. Both algorithms are implemented in Maple (2016) and  {\Blue were executed  on Apple iMac  (Intel i 7-2600, 3.4 GHz, 16GB).} 
The inputs polynomial vectors were generated as follows. The coefficients were  randomly taken  from  $[-10,10]$.   The degrees~$d$ of the polynomials  ranged from $3$ to $15$. The  length~$n$ of the vectors also ranged from $3$ to $15$.  

Figure~\ref{fig-timing} shows the timings.
\begin{figure}[h!] \centering
\includegraphics[scale=0.40]{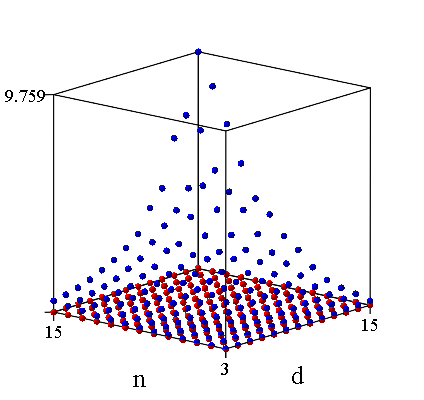} 
\caption{Timing comparison: {\red OMF} vs.~{\blue Two-step approach}}
\label{fig-timing}
\end{figure}
The horizontal axes correspond to $n$ and $d$
and the vertical axis corresponds to computing time $t$ in seconds.
Each dot  $(d,n,t)$ represents an  experimental timing.     The {\red red} dots indicate the experimental timing of the algorithm OMF, while the {\blue blue} dots indicate the experimental timing of the two-step approach described above.

As can be seen, the  algorithm OMF runs  significantly more efficiently. This is due primarily to the cost of computing the Popov form of the last $n-1$ columns of the non-optimal moving frame. As described in \cite{beckermann-labahn-villard-2006}, the complexity of this step is $O(d^3n^7)$, which is bigger than $O(d^{2}n + d^{3} + n^{2})$, the complexity of the OMF (Proposition \ref{prop-complexity}).
Although other algorithms and implementations for Popov form computations may be more efficient than the one  currently implemented  in Maple, we still expect OMF to significantly outperform any  similar two-step procedure, because the degree reduction step is essentially similar to  a TOP reduced Gr\"obner basis computation for a module, which is computationally expensive.



\section{Geometric interpretation and equivariance}
\label{sect-geom-mf}
In the introduction, we justified the term moving frame by picturing  it as a coordinate system moving along a curve.  This point of view is reminiscent  of classical geometric frames, such as the Frenet-Serret frame.  However, the frames in this paper were defined by suitable algebraic properties,  not its  geometric properties. 
It is then natural to ask if it  is possible to combine  algebraic properties of Definition~\ref{def-omf} with some essential geometric properties, in particular with the group-equivariance property.   In this section, we show that any deterministic algorithm for computing an optimal moving frame can be augmented to  obtain an  algorithm that computes a $GL_n(\k)$-equivariant moving frame.

The group-equivariance property is essential for the majority of frames arising in differential geometry. For the Frenet-Serret frame it is manifested as follows. We recall that for a smooth curve $\gamma$ in $\RR^3$,   the Frenet-Serret frame at a point $p \in\gamma$ consists of   the unit tangent vector $T$,  the unit normal vector  $N$  and  the unit binormal  vector $B$ to the curve at $p$. Consider the action of Euclidean group $E(3)$
(consisting of rotations, reflections, and translations) on $\RR^3$. This action induces and action of the curves in $\RR^3$ and on the vectors. It  is easy to see that, for any $g\in E(3)$,  the vectors $g\,T$, $g\,N$ and $g\,B$  are the unit tangent,  the unit normal and  the unit binormal, respectively, at the point $g\,p$   of the curve  $g\,\gamma$. Thus, if we define $F_\gamma(p)=[T, N, P]$, then we  can record  the equivariance property as:
\beq\label{eq-fsf} F_{g\,\gamma}(g\,p)=g\, F_\gamma(p)\quad \text{ for all } \gamma\subset\RR^3, p\in\gamma \text{ and }g\in E(3).\eeq

 In the case of  the algebraic moving frames considered in this paper, we are interested in developing an algorithm that   for $\aa\in \k[s]^n\backslash\{\bf 0\}$  produces an optimal moving frame $P_\aa$  (recall  Definition~\ref{def-omf})  with the additional $GL_n(\k)$-equivariance property:
 \beq\label{eq-eomf} P_{\aa\, g}(s)=g^{-1}\,\mf_\aa(s)  \text{ for all } \aa\in  \k[s]^n\backslash\{{\bf 0}\}, s\in\k \text{ and } g\in GL_n(\k).  \eeq

We observe that on the right-hand side  of \eqref{eq-fsf} the frame is multiplied by $g$, while on the right-hand side  of \eqref{eq-eomf} the frame is multiplied by $g^{-1}$. This means that the columns of $P$ comprise a \emph{right} equivariant moving frame, while  the Frenet-Serret frame is a \emph{left} moving frame (see Definition 3.1 in  \cite{FO99} and the subsequent discussion).  

To give a precise  definition of a $GL_n(\k)$-right-equivariant algebraic moving frame algorithm,  consider  the set $M= \k\times\left(\k[s]^n\backslash\{\bf 0\}\right)$, and   view  an algorithm producing an algebraic moving frame as a map $\rho\colon M\to GL_n(\k)$ such that, for a fixed $\aa$, the matrix $P_\aa(s)=\rho(s,\aa)$ is polynomial in $s$ and satisfies Definition~\ref{def-omf}. Then the $GL_n(\k)$-property  \eqref{eq-eomf} is equivalent to the commutativity of the   following diagram:
 \begin{center}
\begin{tikzcd}
GL_n(\k) \arrow[r, "L_g^{-1}"]  & GL_n(\k)\\
M \arrow[r, "g"] \arrow[u, "\rho"] &  M  \arrow[u, "\rho"]
\end{tikzcd}
\end{center}
On the top  of the diagram, $L_g^{-1}$ indicates the right action of   $g\in GL_n(\k)$ on  $ GL_n(\k)$ defined  by  multiplication from the left by $g^{-1}$, while on the bottom   the  right action is defined by $g\cdot (s,\aa)=  (s,\aa\, g)$.

We observe further that if  the columns of $P$ comprise a {right} equivariant moving frame, then the rows of  $P^{-1}$ comprise a left frame. The  inverse  algebraic frame has an easy geometric interpretation: the first  row of $P^{-1}_\aa$ equals to the position vector $\aa$ and together with the last $n-1$ rows  forms an $n$-dimensional parallelepiped whose volume does not change as the frame moves along the curve.

 It is easy to find an instance of $g$ and $\aa$  to show that  $P_\aa=OMF(\aa)$, where $OMF(\aa)$ is produced  by Algorithm~\ref{alg-OMF}, does not satisfy \eqref{eq-eomf} and, therefore, the OMF algorithm is not a $GL_n(\k)$-equivariant algorithm. However, for input vectors $\aa=[a_1,\dots,a_n]$ such that $a_1,\dots, a_n$ are independent over  $\k$, the  OMF algorithm can be augmented into a  $GL_n(\k)$-equivariant algorithm as follows:

\begin{algorithm}[\textbf{EOMF}]\hfill
\label{alg-eomf}
\begin{description}
\item[\textit{Input:}] $\aa=[a_1,\dots,a_n] \not = 0 \in\mathbb{K}[s]^{n}$, row vector,
where $n>1$, $\gcd(\mathbf{a})=1$, $\mathbb{K}$ a computable field, and components of $\aa$ are linearly independent over $\k$. 

\item[\textit{Output:}] $P\in\mathbb{K}[s]^{n\times n}$, a degree-optimal moving
frame at $\mathbf{a}$
\end{description}

\begin{enumerate}
\item \emph{Construct an $n\times n$ invertible submatrix  of the coefficient matrix of $\aa$.} 

\begin{enumerate}
\item $d\longleftarrow\deg(\mathbf{a})$

\item Identify the row vectors $c_{0}=[c_{01}, \dots c_{0n}],\ldots,c_{d}=
[c_{d1}, \dots c_{dn}]$ such that $\mathbf{a}=c_{0}+c_{1}s+\cdots+c_{d}s^{d}$.

\item $I=[i_1,\dots, i_n]\longleftarrow$ lexicographically smallest vector of integers between $0$ and $d$, such that vectors $c_{i_1},\dots, c_{i_n}$ are independent over $\k$.
\item $\widehat C\longleftarrow
\left[
\begin{array}{c}
  c_{i_1}   \\
  \vdots \\
  c_{i_n}
\end{array}
\right]
$
\end{enumerate}
\item \emph{Compute an optimal moving frame for a canonical representative  of the $GL_n(\k)$-orbit of $\aa$.}
$$\widehat P\longleftarrow OMF(\aa \,\widehat C^{-1})$$
\item \emph{Revise  the moving frame  $\widehat P$ so that the algorithm has the equivariant property \eqref{eq-eomf}}. 
$$\mf \longleftarrow \widehat C^{-1}\widehat P.$$
\end{enumerate}
\end{algorithm}
To prove the algorithm we need the following proposition.
\begin{proposition}\label{prop-equiv} Let $\mf $ be a degree-optimal moving frame at a nonzero polynomial vector $\aa$. Then, for any $g\in GL_n(\k)$, the matrix $g^{-1}\mf$ is a degree-optimal moving frame at the vector $\aa\, g$.
\end{proposition}
\begin{proof}
By definition, $\aa\, \mf=[\gcd(\aa), 0,\dots,0]$ and, therefore,  for any $g\in GL_n(\k)$ we have: 
$$(\aa\,g)\,g^{-1} \mf= [\gcd(\aa), 0,\dots,0].$$
From this, we conclude that $\gcd(\aa\,g)=\gcd(\aa)$ and that $g^{-1} \mf$ is a moving frame at $\aa\,g$.  We note that the rows of the matrix $g^{-1} \mf$ are linear combinations over $\k$ of the rows of the matrix $\mf$. Therefore, the degrees of the columns  of $g^{-1} \mf$ are less than or equal to the degrees of the corresponding columns of $\mf$.

Assume that  $g^{-1} \mf$  is not a degree-optimal moving frame at  $\aa\,g$. Then there exists a moving frame $\mf'$ at $\aa\,g$ such that at least one of the columns of  $\mf'$, say the $j$-th column, has degree strictly less than the $j$-th column of $g^{-1} \mf$. Then, from the paragraph above, the $j$-th column of   $\mf'$ has degree strictly less than the degree of the $j$-th column of $\mf$. 

By the same argument, $g\,\mf'$ is a moving frame at $\aa$ such that its $j$-th column has degree less than or equal to the degree of the $j$-th column of $\mf'$, which is strictly less than the degree of the $j$-th column of $\mf$. This contradicts our assumption that $\mf$ is degree-optimal.
 \end{proof}
\emph{Proof of the Algorithm~\ref{alg-eomf}.}
We first note that, since polynomials $a_1,\dots,a_n$ are linearly independent over $\k$, then  the coefficient matrix $C$ contains $n$ independent rows and, therefore,  Step 1 of the algorithm can be accomplished.
 Let  $\wh\aa=\aa \,\widehat C^{-1}$, then $\aa=\wh\aa\,\wh C$ and $P$ is an optimal moving frame at $\aa$ by Proposition~\ref{prop-equiv}. To show \eqref{eq-eomf},  for an arbitrary input $\aa_1$ and an arbitrary $g\in GL_n(\k)$, let $\aa_2=\aa_1\, g$.  Then $\wh C_{\aa_2}=\wh C_{\aa_1}\,g$ and so
$$ EOMF(\aa_2)= \wh C_{\aa_2}^{-1} OMF(\aa_2 \,\wh C_{\aa_2}^{-1})=g^{-1}\wh C_{\aa_1}^{-1} OMF(\aa_1\,g \,g^{-1}\,\wh C_{\aa_1}^{-1})=g^{-1}\,EOMF(\aa_1).$$
\qed

\begin{remark} It is clear from the above proof that if, in Step~2 of Algorithm~\ref{alg-eomf},  we replace OMF with \emph{any} (not necessarily degree-optimal) algorithm,  then (not necessarily degree-optimal)   frames produced by  Algorithm~\ref{alg-eomf} will have the $GL_n(\k)$-equivariant property \eqref{eq-eomf}.

\end{remark}


\vskip5mm

\section*{Acknowledgments}  
{\Blue  We are grateful to  David Cox for posing the question about the relationship between the degree of the minimal \bez vector and the $\mu$-type which led to Proposition~\ref{mu-type} of this paper; to Teresa Krick for the discussion of the Quillen-Suslin theorem; and to George Labahn for the discussion of the degree-reduction algorithms and the Popov normal form.}

\bibliographystyle{plain}
\bibliography{paper}

\end{document}